\documentclass[11pt]{article}

\usepackage{mathtools, amsmath, amsthm, amsfonts,amssymb, mathrsfs, enumitem, graphicx, colortbl, tikz, bbm, dsfont, wrapfig, esint, subfig,float, graphics,epsfig,psfrag, color,soul}

\usepackage[utf8]{inputenc}
\usepackage[T1]{fontenc}
\usepackage[active]{srcltx} 
\usepackage{comment, microtype}
\DisableLigatures{encoding = *, family = * }

\newtheorem{theorem}{Theorem}[section]
\newtheorem{corollary}[theorem]{Corollary}
\newtheorem{lemma}[theorem]{Lemma}
\newtheorem{prop}[theorem]{Proposition}
\theoremstyle{definition}
\newtheorem{definition}[theorem]{Definition}
\newtheorem{remark}{Remark}

\DeclareMathOperator\sign{sign}
\DeclareMathOperator\supp{supp}

\def\N{\mathbb{N}}

\def\R{\mathbb{R}}

\def\ind{\mathbbm{1}}

\let\O=\Omega
\let\e=\varepsilon
\let\ve=\varepsilon

\let\t=\tilde

\let\.=\cdot
\let\0=\emptyset

\let\mc=\mathcal

\let\O=\Omega
\let\e=\varepsilon

\let\t=\tilde

\let\.=\cdot
\let\0=\emptyset

\let\mc=\mathcal

\def\1{\mathbbm{1}}

\def\t f{\tilde{f}}

\def\d{\,\mathrm{d}}
\def\ve{\varepsilon}
\def\ind{\mathbbm{1}}
\newcommand{\deb}{\rightharpoonup}

\newenvironment{formula}[1]{\begin{equation}\label{#1}}
	{\end{equation}\noindent}

\def\Fi#1{\begin{formula}{#1}}
	\def\Ff{\end{formula}\noindent}

\usepackage[colorlinks=true,linkcolor=blue,citecolor=blue,urlcolor=blue,breaklinks]{hyperref}

\usepackage{hyperref}
\usepackage{breakurl}
\usepackage[square,sort,comma,numbers]{natbib}
\usepackage{url}
\definecolor{lpink}{rgb}{0.96, 0.76, 0.76}
\definecolor{dpink}{rgb}{0.97, 0.51, 0.47}
\definecolor{sky}{rgb}{0.53, 0.81, 0.92}
\definecolor{salmon}{rgb}{1.0, 0.55, 0.41}
\definecolor{orman}{rgb}{0.24, 0.7, 0.44}
\definecolor{aciksari}{rgb}{0.91, 0.84, 0.42}
\definecolor{dgrey}{rgb}{0.52, 0.52, 0.51}
\definecolor{ao}{rgb}{0.0, 0.5, 0.0}
\definecolor{purple}{rgb}{0.6, 0.4, 0.8}
\definecolor{purple2}{rgb}{0.63, 0.36, 0.94}
\definecolor{purpleheart}{rgb}{0.41, 0.21, 0.61}

\usepackage{lipsum}

\newcommand\blfootnote[1]{%
	\begingroup
	\renewcommand\thefootnote{}\footnote{#1}%
	\addtocounter{footnote}{-1}%
	\endgroup
}

\def\R{\mathbb{R}}

\def\Slo{\mathrm{Slope}F}

\def\d{\,\mathrm{d}}
\def\O{\Omega}

\def\p{\,\partial}

\DeclareMathOperator{\argmin}{argmin}

\hypersetup{pdftitle={CrossDiffGradFlow}}
\hypersetup{pdfauthor={Romain Ducasse  \& Filippo Santambrogio  \& Havva Yolda\c{s}}}

\author{
	Romain Ducasse\footnote{Université de Paris and Sorbonne Université, CNRS, Laboratoire Jacques-Louis Lions (LJLL), F-75006 Paris, France. ducasse@math.univ-paris-diderot.fr}
	\and  Filippo Santambrogio  \footnote{Institut Camille Jordan, Université Claude Bernard Lyon 1, 43 bd du 11 novembre 1918, 69622 Villeurbanne Cedex, France. santambrogio@math.univ-lyon1.fr \& Institut Universitaire de France}
	\and Havva Yoldaş \footnote{Delft Institute of Applied Mathematics, Faculty of Electrical Engineering, Mathematics and Computer Science, Delft University of Technology, Mekelweg 4, 2628CD Delft, Netherlands. h.yoldas@tudelft.nl}}

\title{A cross-diffusion system obtained via (convex) relaxation in the JKO scheme}
\makeatletter
\makeindex
\date{}
\begin{document}
	\maketitle
	\begin{abstract}
		\noindent
		In this paper, we start from a very natural system of cross-diffusion equations, which can be seen as a gradient flow for the Wasserstein distance of a certain functional. Unfortunately, the cross-diffusion system is not well-posed, as a consequence of the fact that the underlying functional is not lower semi-continuous. We then consider the relaxation of the functional, and prove existence of a solution in a suitable sense for the gradient flow of (the relaxed functional). This gradient flow has also a cross-diffusion structure, but the mixture between two different regimes, that are determined by the relaxation, makes this study non-trivial. 
		
	\end{abstract}

	
	\section{Introduction}
	\label{sec:introduction}
	
	\blfootnote{\emph{Key words and phrases.} cross-diffusion, gradient flow, Wasserstein distance, JKO scheme, convex relaxation}
	\blfootnote{\emph{2010 Mathematics Subject Classification.} 35A01; 35A15; 49J45}
	The starting point of this paper is the following very natural system of PDEs
	
	\begin{align} \label{eq:cross_diff_system}
		\begin{split}
			\begin{cases}
				\p_t \rho  = \Delta \rho + \nabla \cdot (\rho \nabla \mu), \quad t>0,\ x\in \O,  \\
				\p_t \mu  = \Delta \mu+ \nabla \cdot (\mu \nabla \rho),\quad t>0,\ x\in \O,  
			\end{cases}
		\end{split}
	\end{align} complemented with no-flux boundary conditions in a bounded domain $\O$. This cross-diffusion system describes the motion of two populations, each subject to diffusion and trying to avoid the presence of the other, so that each density acts as a potential in the evolution equation satisfied by the other.
	
	\medskip
	
	This model was studied in \cite{BRYZ21}, where it is obtained as a continuum version of a discrete lattice model proposed in \cite{AB18}, to account for the territorial development of two competing populations. Similar models appear in mathematical biology to describe the evolution of interacting species that are under the influence of population pressure due to intra- and inter-specific interferences (see e.g. \cite{SKT79, GS14}). However, in these papers, the models considered always enjoy some special structure that ensures some ``convexity'', which is crucially not verified by System \eqref{eq:cross_diff_system}, as we will explain later.
	
	\medskip

	The existence of solutions for \eqref{eq:cross_diff_system} is a very challenging problem. An easy computation shows that the system is parabolic when $\rho\mu<1$ but has an anti-parabolic behavior when $\rho\mu>1$. Existence of solutions for short time is proven by the third author and her collaborators in \cite{BRYZ21}, under the assumption that the initial data satisfy $\rho_0\mu_0<1$. 
	
	\medskip
	A noticeable property of system \eqref{eq:cross_diff_system} is the following: it can be seen as a \emph{gradient flow} for a suitable functional in the Wasserstein space.
	
	\medskip
	
	The usual notion of gradient flows applies to Hilbert spaces but in the last two decades the interest has grown for gradient flows in metric spaces and in particular in the space of probability measures endowed with the Wasserstein distance (see \cite{AGS08} and \cite{S17}), after the seminal work by Jordan, Kinderlehrer and Otto \cite{JKO98}, who found a gradient flow structure for the Fokker-Planck equation. Applying the same ideas not to a single PDE but to a system of PDEs, thus looking for gradient flows in the space of pairs of probability measures, is more recent (see \cite{CFSS18,KM18}) and more delicate.
	
	\medskip
	At a formal level, given a functional $F$ defined on probabilities on a given domain $\Omega$, its gradient flow corresponds to the evolution PDE 
	$$\partial_t\rho=\nabla\cdot \Big (\rho\nabla \Big (\frac{\delta F}{\delta \rho}\Big ) \Big ), $$
	where $\delta F/\delta \rho$ is the first variation of the functional $F$ (formally defined through the condition $F(\rho+\ve\delta\rho)=F(\rho)+\ve\int \frac{\delta F}{\delta\rho}\d\delta\rho+o(\ve)$). This equation is endowed with no-flux boundary conditions on $\partial \Omega$. Analogously, given a functional $F$ defined on pairs $(\rho,\mu)$ of probabilities on the domain $\Omega$, the gradient flow of $F$ in $W_2(\Omega)\times W_2(\Omega)$ would be given by the system
	$$\begin{cases}
		\partial_t\rho=\nabla\cdot (\rho\nabla(\frac{\delta F}{\delta \rho})),\\
		\partial_t\mu=\nabla\cdot(\mu\nabla(\frac{\delta F}{\delta \mu})),
	\end{cases}$$
	with again no-flux on the boundary. Of course, we formally define the two partial first variations via the condition $F(\rho+\ve\delta\rho, \mu+\ve\delta\mu)=F(\rho,\mu)+\ve\int \frac{\delta F}{\delta\rho}\d\delta\rho+\ve\int \frac{\delta F}{\delta\mu}\d\delta\mu+o(\ve)$.
	
	\medskip
	
	Here, it turns out that \eqref{eq:cross_diff_system} is the gradient flow of the following functional
	\begin{align} \label{defn:F_0}
		F_0(\rho, \mu ) =  \begin{cases}
			\int_{\Omega} f_0(\rho(x),\mu(x))\d x \quad &\mbox{ if } \rho,\mu \in L^1_+(\O),\\
			+\infty \quad &\mbox{ otherwise},
		\end{cases} 
	\end{align}
	where the function $f_0$ is given by $f_0(a,b)=a \log a + b\log b +ab$, and $\rho(x)$, $\mu(x)$ stand for the densities of the two measures, which are supposed to be absolutely continuous and identified with their $L^1$ densities.
	
	\medskip
	
	The main problem in considering the function $f_0$ is that it is not globally convex, as we have, denoting $D^2f_0$ the Hessian of $f_0$,
	$$D^2f_0(a,b)=\left(\begin{array}{cc}\frac 1a&1\\1&\frac 1b\end{array}\right),$$
	and this matrix is only positive-definite if $ab<1$. The non-convexity of $f_0$ translates into the non-parabolic behavior of the system outside the region where $\rho\mu<1$.
	
	\medskip
	
	Yet, in terms of the functional $F_0$, the situation is even worse. Indeed, integral functionals with a non-convex integrand are not lower semi-continuous for the weak convergence of probability measures. This is a non-negligible issue when applying variational methods. Indeed, one of the main tools to study gradient flows in a metric setting is the so-called {\it minimizing movement} method (see \cite{DeG93,Amin}). Let us quickly explain this tool in the general case of a metric space, and we specify it to the case of interest here just after.
	
	\medskip

	For a given functional $F$, acting on a metric space $M$ endowed with a distance $d$, the minimizing movement scheme consists in building an approximation of the gradient flow (which is a curve $x(t)$ in $M$) as follows: we fix a time step $\tau>0$ and look for a sequence $x^\tau_k$ such that
	$$x^\tau_{k+1}\in \argmin_x F(x)+\frac{d(x,x^\tau_k)^2}{2\tau}.$$

	In the case where the metric space is the space of propability measures on $\Omega$ and the distance $d$ is the Wasserstein distance, this iterated minimization scheme is known under the name of {\it Jordan-Kinderlehrer-Otto} (JKO) scheme. It produces a sequence of probability measures which, when appropriately interpolated, would approximate the gradient flow of $F$.
	
	\medskip
	
	Here, we will consider the case of a gradient flow in $M = W_2(\Omega)\times W_2(\Omega)$ (and the distance $d$ is the natural product of the Wasserstein distance). The minimizing movement method then consist in solving the following family of minimization problems:
	\begin{equation}\label{JKOsquare}
		(\rho^\tau_{k+1},\mu^\tau_{k+1})\in\argmin_{(\rho,\mu)}F(\rho,\mu)+\frac{W_2^2(\rho,\rho^\tau_k)+W_2^2(\mu,\mu^\tau_k)}{2\tau}.
	\end{equation}
	However, if $F$ is not lower semi-continuous for the weak convergence, which is the case when $F=F_0$ defined above, while the terms in $W_2^2$ are continuous, the existence of a minimizer in the above iterated minimization problem is not always guaranteed. From the variational point of view, it is necessary to replace the functional $F_0$ with its {\it relaxation}, or lower semi-continuous envelope.
	
	\medskip

	We stress that the relation between the gradient flow of a functional and its lower semi-continuous envelope is not clear in terms of PDEs, but it is clearer at the level of the JKO scheme. Indeed, if we fix $\tau>0$ and a tolerance parameter $\delta>0$, when facing a non-lower semi-continuous functional $F$ (whose lower semi-continuous envelope is denoted by $\bar F$), we could define a $\delta$-approximated JKO scheme in the following way: we denote by $\argmin_x^\delta G(x)$ the set of $\delta$-almost minimizers of $G$, i.e. the set $\{x\,:\, G(x)<\inf G +\delta\}$, which is always non-empty, and we pick any sequence $x^{\tau,\delta}_k$ satisfying
	$$x^{\tau,\delta}_{k+1}\in \argmin_x^\delta F(x)+\frac{d(x,x^{\tau,\delta}_k)^2}{2\tau}.$$
	It can be proven, as an application of $\Gamma-$convergence, that, for fixed $\tau$ and letting $\delta\to 0$, any such sequence converges to a sequence satisfying
	$$x^\tau_{k+1}\in \argmin_x \bar F(x)+\frac{d(x,x^\tau_k)^2}{2\tau},$$
	i.e. to the output of the JKO scheme for the lower semi-continuous relaxation $\bar F$.
	
	\medskip 
	
	The present paper is then devoted to the study of the system of PDEs representing the gradient flow of a functional $F$ obtained as the lower semi-continuous envelope of $F_0$, and in particular to an existence result. Let us detail the structure of the paper. As a first task, we need to compute this lower semi-continuous envelope. The general theory about local functionals on measures (see \cite{DeGiorgiLSC}) provides a clear answer: one just needs to replace the non-convex function $f_0$ with its convex envelope, that we will denote by $f$. We therefore compute the convex envelope $f$ of $f_0$ and find that it has the following form: on a certain region $B\subset\R_+^2$ we have $f=f_0$, and on the complement $A=\R_+^2\setminus B$ the function $f$ only depends on the sum of the two variables, i.e. there exists $\tilde f:\R\to\R$ such that $f(a,b)=\tilde f(a+b)$ for $(a,b)\in A$. 
	\medskip
	
	The fact that $f$ partially agrees with a function of the sum only recalls some cross-diffusion problems already studied in other papers using Wasserstein techniques, such as in  \cite{CFSS18,KM18}. Cross-diffusion with a functional only depending on the sum is not difficult to study, as one can first find an equation on the sum $S=\rho+\mu$ and then use the gradient of the solution $S$ as a drift advecting both $\rho$ and $\mu$. The difficulty in most of the cross-diffusion problems involving the sum comes from the lower-order terms which differentiate the two species (reaction, or advection), since the above strategy cannot be applied and in general it is not possible to obtain compactness results on the two densities $\rho$ and $\mu$ separately, but only on their sum. However, when the functional involves the integral of a strictly convex function of the two densities $\rho$ and $\mu$, it is indeed possible to obtain separate compactness: here the challenge comes from combining the two regimes, one where $f$ is strictly convex and one where it only depends on the sum, but without extra terms differentiating the two species.
	
	\medskip
	
	On the other hand, a remark is compulsory when looking at the precise definition of the function $f$, since the region $B$ where $f=f_0$ is strictly included in the set $\{ab<1\}$. Indeed, one could think (and it would have been nice!) that the system we obtain is an extension of the one with $f_0$ (that is, \eqref{eq:cross_diff_system}), in the sense that it allows to extend the solutions even after touching the dangerous curve where $\rho\mu=1$. Unfortunately, this interpretation is not correct since there are initial data satisfying $\rho\mu<1$ but $(\rho,\mu)\notin B$, for which the two systems would have well-defined different solutions, at least for short time.

	As a result, this paper will only be concerned with the gradient flow of the new functional $F$ (the lower semi-continuous envelope of $F_0$, which is, in this case, its convexification as well), without discussing its relation with the original PDE which motivated the study. For this new gradient flow we will prove existence of solutions in a suitable sense.
	
	\medskip
	The notion of solution we consider is inspired by the notion of \emph{EDI solution} introduced in \cite{Amin,AGS08} in terms of the metric slope, but it is slightly different and more PDE-adapted. Given our functional $F$, our goal is to find a solution of the following system
	\begin{equation}\label{weaksystem}
		\begin{cases}
			&\partial_t\rho+\nabla\cdot(\rho v)=0,\\
			&\partial_t\mu+\nabla\cdot (\mu w)=0,\\
			&v=-\nabla\left(\frac{\delta F}{\delta\rho}\right)\quad \rho-a.e.,\\
			&w=-\nabla\left(\frac{\delta F}{\delta\mu}\right) \quad \mu-a.e.,
		\end{cases}
	\end{equation}
	where the two continuity equations are satisfied in a weak sense with no-flux boundary conditions on $\partial\Omega$ and the functions $\frac{\delta F}{\delta\rho}$ and $\frac{\delta F}{\delta\mu}$ are differentiable in a suitable weak sense.
	In our particular case this means finding a pair $(\rho_t,\mu_t)$ such that the gradients of $f_a(\rho_t,\mu_t)$ and $f_b(\rho_t,\mu_t)$ exist in such a sense, and the above equations are satisfied.

	Formally, if the two first conditions of \eqref{weaksystem} are satisfied, then the last two conditions are equivalent, on the interval $[0,T]$, to the following inequality
	\begin{equation}\label{EDIformulation}
		F(\rho_T,\mu_T)+\frac12\int_0^T\int_\Omega \rho \left |\nabla\frac{\delta F}{\delta\rho} \right |^2+
		\frac12\int_0^T\int_\Omega \mu \left |\nabla\frac{\delta F}{\delta\mu} \right |^2+
		\frac12\int_0^T\int_\Omega \rho|v|^2+\frac12\int_0^T\int_\Omega \mu|w|^2\leq F(\rho_0,\mu_0).
	\end{equation}
	Indeed, if we formally differentiate the function $t\mapsto F(\rho_t,\mu_t)$ in time, we obtain
	$$\frac{\d}{\d t}F(\rho_t,\mu_t)=\int_\Omega \rho \nabla\frac{\delta F}{\delta\rho}\cdot v+\int_\Omega \mu \nabla\frac{\delta F}{\delta\mu}\cdot w,  $$
	hence 
	\begin{equation}\label{diffprop}
		F(\rho_0,\mu_0)=F(\rho_T,\mu_T)-\int_0^T\int_\Omega \rho \nabla\frac{\delta F}{\delta\rho}\cdot v-\int_0^T\int_\Omega \mu \nabla\frac{\delta F}{\delta\mu}\cdot w.
	\end{equation}
	This means that \eqref{EDIformulation} can be re-written as
	$$\frac12 \int_0^T\int_\Omega\rho \left |\nabla\frac{\delta F}{\delta\rho}+v \right |^2+\frac12 \int_0^T\int_\Omega\mu \left |\nabla\frac{\delta F}{\delta\mu}+w \right |^2\leq 0,$$
	and the only way to satisfy this condition is to satisfy the a.e. equality of the last two lines in \eqref{weaksystem}.
	
	\medskip 
	This trick to characterize the solutions comes from the Euclidean observation that $x'(t)=-\nabla F(x(t))$ is equivalent to $\frac{\d}{\d t}F(x(t))\leq -\frac12 |x'(t)|^2-\frac12|\nabla F(x(t))|^2$, and hence to the {\it Energy Dissipation Inequality } (EDI)
	$$F(x(T))+\frac12\int_0^T |x'(t)|^2+\frac12\int_0^T|\nabla F(x(t))|^2\leq F(x_0).$$
	The fact that gradients and derivatives cannot be defined in metric spaces (a vector structure is needed) but their norms could be defined (using the so-called metric derivative and metric slope) instead leads to the definition of a notion of \emph{EDI gradient flow in metric spaces}. This is what is done in \cite{AGS08} for general metric spaces (in particular for functionals $F$ which are geodesically convex) and then particularized to the case of the Wasserstein space.
	
	However, this is not the strategy which is followed in our paper, even though what we do is strongly inspired by the metric approach of \cite{AGS08}. Our precise procedure is the following:
	\begin{itemize}
		\item We define a class $\mathcal H$ consisting of pairs $(\rho,\mu)$ where $\nabla f_a(\rho,\mu)$ and $\nabla f_b(\rho,\mu)$ are well-defined. We do this by requiring that some functions of the densities $(\rho,\mu)$ belong to the $H^1$ Sobolev space: in particular, we require that $\eta(\rho,\mu)$ is $H^1$ for any smooth function $\eta$ supported in the set $B$ where $f$ has a strictly convex behavior, and we also require $\sqrt{\rho+\mu}$ to be $H^1$. The second requirement is not sharp, in the sense that other functions of the sum could be used as well. We chose to use this one for simplicity, since we guarantee this condition on the solutions which we build via extra arguments.
		\item We define, on the class $\mc H$, the slope functional $\Slo$ via 
		\begin{equation}\label{roughdefislo}
			\Slo(\rho,\mu):=\int_\Omega \rho|\nabla f_a(\rho,\mu)|^2+\mu|\nabla f_b(\rho,\mu)|^2.
		\end{equation}
		Note that we do not pretend at all that this slope functional coincides with the metric slope which could be defined by following the theory in the first part of the book \cite{AGS08}.
		\item We say that a pair of curves $(\rho,\mu)$ is an EDI solution if $(\rho_t,\mu_t)\in\mathcal H$ for a.e. $t$ and we have
		\begin{equation}\label{EDIsol}
			F(\rho_T,\mu_T)+\frac12\int_0^T\Slo(\rho,\mu)+
			\frac12\int_0^T\int_\Omega \rho|v|^2+\frac12\int_0^T\int_\Omega \mu|w|^2\leq F(\rho_0,\mu_0),
		\end{equation}
		for some velocity fields $v,w$ solving the continuity equations $\partial_t\rho+\nabla\cdot(\rho v)=0, \, \partial_t\mu+\nabla\cdot (\mu w)=0$.
	\end{itemize}
	
	In order to prove the existence of an EDI solution we rely on the JKO scheme \eqref{JKOsquare} and build suitable interpolations in time of the sequence of solutions, thus obtaining an approximation of \eqref{EDIsol}. More precisely, we will have 
	$$ F(\rho_T,\mu_T)+\frac12\int_0^T\Slo(\hat\rho,\hat\mu)+
	\frac12\int_0^T\int_\Omega \tilde\rho|\tilde v|^2+\frac12\int_0^T\int_\Omega \tilde\mu|\tilde w|^2\leq F(\rho_0,\mu_0)$$
	for two different interpolations $(\hat\rho,\hat\mu)$ and $(\tilde\rho,\tilde\mu)$. We then pass to the limit as $\tau\to 0$ ($\tau$ is the time step for the discretization in the JKO scheme) where the weak limits of the different interpolations coincide. We prove that $\Slo$ is a lower semi-continuous functional for the weak convergence, which allows us to conclude, combined with more standard semi-continuity arguments.
	
	\paragraph{Organization of the paper} After this introduction, Section \ref{sec:convexification} is devoted to the computation of the convexification of $f_0$, and to some properties of the function $f$ we obtain, introducing some relevant quantities. Section \ref{sec l.s.c.} is devoted to the precise definition of the slope $\Slo$ and to the proof of its lower semi-continuity. Section \ref{sec:EDI} introduces the notion of EDI solutions and proves their existence. In the proof, several interpolations of the sequence obtained via the JKO scheme are needed, including the De Giorgi {\it variational interpolation}. Some estimates on these solutions are required in order to prove that they belong to the class $\mathcal H$ and to obtain the desired inequalities. Sections \ref{sec:differentiation_properties} and \ref{sec:H1_estimate} are not required to obtain the existence of EDI solutions, but are required if one wants to come back to System \eqref{weaksystem}. Indeed, we stated that the notions are formally equivalent thanks to an easy computation for the derivative in time of $F(\rho_t,\mu_t)$. However, this computation is only formal if we do not face smooth solutions. This explains the choice of the notion of EDI solutions: it is a definition which coincides with solving the equation in a classical sense if functions are smooth, but the equivalence is in general not granted. In Section \ref{sec:differentiation_properties}, we then explain an approximation procedure (by convolution) to obtain the differentiation property (i.e. the validity of \eqref{diffprop}) for non-smooth solutions of the continuity equation. Yet, the nonlinearity of the terms involved in $f_a$ and $f_b$ requires a certain bound on the $H^1$ norm of some functions of the regularized functions. This raises a very natural question: suppose that a function $u$ is such that its positive part $u_+$ belongs to $H^1$, and let $u_\ve$ be its convolution with a certain mollifier $\eta_\ve$: when is it true that the sequence $(u_\ve)_+$ is bounded in $H^1$ (with, possibly, uniform bounds in terms of the kernel)? This is not trivial and the answer could depend on the choice of the kernel. We provide in Section \ref{sec:H1_estimate} a proof of this fact in dimension $1$ for a specific choice of the kernel shape. Note that the very same result is false (even after changing the shape of the kernel) in higher dimensions (we thank Alexey Kroshnin for a refined counter-example in this direction). Nevertheless, this does not prevent the approximation or the differentiation to be true.

	\section{Convexification}
	\label{sec:convexification}
	
	In this section, we characterize the function $f$ which is the convex envelope of the function $f_0$, and $f_0$ is defined by $f_0(a,b)=a\log a + b\log b +ab$ on $\R_+^2$. The functional $F$ defined through
	\begin{align} \label{defn:F}
		F(\rho, \mu ) =  \begin{cases}
			\int_{\Omega} f(\rho(x),\mu(x) )\d x \quad &\mbox{ if } \rho,\mu \in L^1_+(\O),\\
			+\infty \quad &\mbox{ otherwise,}
		\end{cases} 
	\end{align}
	will be the lower semi-continuous envelope of $F_0$ (which is itself defined in a similar way replacing $f$ with $f_0$) for the weak convergence of probability measures, thanks to standard relaxation results (see, for instance, \cite{DeGiorgiLSC}).
	
	As we underlined in the introduction, $f_0$ is not convex since $D^2 f_0(a,b)$ is not positive semi-definite unless $ab \leq 1$.

	First, we look at the shape of $f_0$ on diagonal lines where $a+b$ is constant. We denote $s:=a+b$ and for a fixed $s$, we define a function $g(a) := f_0(a, s-a)$ which gives the value of $f_0$ over a segment for $s$ fixed. Then we have
	\begin{align*}
		g'(a)&=  \log (a) -  \log (s-a) + s-2a,\\
		g''(a)&=  \frac{1}{a} + \frac{1}{s-a} -2.
	\end{align*}
	We need to distinguish the cases where $s\leq 2$ and where $s>2$. In the first case, $g$ is convex, since $ \frac{1}{a} + \frac{1}{s-a}\geq \frac 4s\geq 2$. In the second case, we can easily see that $g$ has a double-well shape, with three critical points on $(0,s)$ (see Figure \ref{fig:convex}) and the critical points are characterized by (using $b=s-a$):
	\begin{align*}
		\log (a)+ b &= \log (b) +a, \\
		a+b &=s.
	\end{align*}
	For $s\leq 2$ only $a=b=s/2$ is a critical point, otherwise the same point is a local maximizer and there are two global minimizers.
	
	Due to the double-well shape of $g$ and to the fact that $g''$ vanishes only twice, the two minimizers of $g$ satisfy $g''>0$ (and not only $g''\geq 0$). Indeed, if this were not the case, then on the interval between the two minimizers, $g$ would be strictly concave and would have vanishing derivative at the endpoints, which is impossible.

	It will turn out that the convexification of $f_0$ on the line $a+b=s$ will coincide with the $1$-dimensional convexification of $g$.
	
	In order to construct such a function $f$, we first define two auxiliary functions $\alpha,\beta$ and two sets $A,B$: indeed, there exist two functions $\alpha\neq\beta \in C^0([2,+\infty))$ characterized by 
	\begin{itemize}
		\item 	$\alpha(2)=\beta(2)=1$,
		\item for every $s>2$ we have $\alpha(s)<s/2<\beta(s)$,
		\item for every $s>2$ we have
		\begin{align} \label{eq:minima_f}
			\begin{split}
				\log (\alpha(s))+ \beta(s)&= \log (\beta(s)) + \alpha(s),  \\
				\alpha(s)+ \beta (s) &= s. 
			\end{split}
		\end{align}
	\end{itemize}
	
	This means that for every $s>2$, the points $\alpha(s)$ and $\beta(s)$ are the two minimizers of $g$ and these conditions are enough to determine $\alpha,\beta$ in a unique way. They can also be obtained using an Implicit Function Theorem, which also shows the smoothness of $\alpha,\beta$. On the other hand, we lose the smoothness at $s=2$ where the IFT cannot be applied, but we have anyway continuity due to the uniqueness of the minimizer.
	
	We then define two sets $A$ and $B$ as those sets that partition $\R_+^2$, with $A$ being the closed, convex set above the curve $\{(\alpha(s),\beta(s)) : \, s \geq 2\}\cup \{(\beta(s),\alpha(s)) : \, s \geq 2\} $, and $B = \R_+^2 \setminus A$ the open set below. 
	The two sets $A$ and $B$ are represented in Figure \ref{fig:domain}.
	
	\begin{figure}[ht!]
		\centering
		\begin{minipage}{0.45\textwidth}
			\centering
			\begin{tikzpicture} [scale=0.6]
				
				\draw[->, very thick ] (-0.2,0) -- (8,0) node[right] {$a$};
				\draw[->, very thick ] (0,-0.2) -- (0,8) node[above] {$b$};
				
				\path[fill=lpink,opacity=1]
				plot[domain=0.355:7.9] (\x,{2*(\x*\x+4*\x+36)^0.5*(\x*\x +4*\x)^(-0.5)-2)}) -- (7.9,7.9) -- cycle;
				\path[fill=aciksari,opacity=1]
				(0,0) --   plot[domain=0.35:7.9] (\x,{2*(\x*\x+4*\x+36)^0.5*(\x*\x +4*\x)^(-0.5)-2)}) --  cycle;
				\path[fill=aciksari,opacity=1]
				(0,0) --  (0.36, 0 ) -- (0.36,7.85) -- (0, 7.85)--  cycle;
				\path[fill=aciksari,opacity=1]
				(0,0) --  (0, 0.36) -- (7.9, 0.36) -- (7.9,0)--  cycle;
				\draw[color=blue, very thick, dashed]    plot[domain=0.51:7.9] (\x,4/\x) node[right] {$ab =1$};
				\draw[color=black, very thick, dashed]    plot[domain=0:-4](\x+4, -\x) node[anchor= east] {$a+b =2$};
				\draw[color=black, very thick, dashed]    plot[domain=0:-6](\x+6, -\x) node[anchor=east] {$a+b =s$};
				\draw[color=red, very thick]   plot[domain=0.36:7.9] (\x,{2*(\x*\x+4*\x+36)^0.5*(\x*\x +4*\x)^(-0.5)-2)}) ;
				\draw[gray, thick,->] (0,0) -- (8,8); 
				\filldraw[fill=purple2, draw=purple2] (3,3) circle (0.13cm);
				\filldraw[fill=purple2, draw=purple2] (0.64,5.38) circle (0.13cm);
				\filldraw[fill=purple2, draw=purple2] (5.38,0.64) circle (0.13cm);
			\end{tikzpicture} 
			\caption{\small Partition of $\R_+^2$}
			\label{fig:domain}
		\end{minipage}\hfill
		\begin{minipage}{0.45\textwidth}
			\centering
			\begin{tikzpicture}[scale=0.8]
				\draw[thick,->] (-1.9,0) -- (3.5,0) node[right] {$a$}; 
				\draw[thick,->] (-1.8,-0.1) -- (-1.8,6.4) ; 
				\draw[thick][scale=0.8, domain=-2.2:2.2, smooth, variable=\x, orange] plot ({\x}, {(1.5-\x)^2*(1.5+\x)^2+1}) node[anchor=north west, ao] {};
				\draw[very  thick,dashed, ao] (-1.2,0.8) -- (1.2,0.8) node[below] {$\tilde f$}; 
				\draw[very  thick, dashed][scale=0.8, domain=1.7:2.2, smooth, variable=\x, ao] plot ({\x}, {(1.5-\x)^2*(1.5+\x)^2+1}) node[anchor=south west, orange] {$g$};
				\draw[very  thick, dashed][scale=0.8, domain=-1.6:-2.2, smooth, variable=\x, ao] plot ({\x}, {(1.5-\x)^2*(1.5+\x)^2+1});
				\filldraw[fill=purple2, draw=purple2] (0,4.8) circle (0.1cm);
				\filldraw[fill=purple2, draw=purple2] (1.2,0.9) circle (0.1cm);
				\filldraw[fill=purple2, draw=purple2] (-1.2,0.9) circle (0.1cm);
				\draw[thin,dashed, black] (0,4.8) -- (0,0) node[below] {$s/2$}; 
			\end{tikzpicture}
			\caption{ \small Functions $g$ and $\tilde{f}$}
			\label{fig:convex}
		\end{minipage}
		\caption{ \small In Figure \ref{fig:domain}, $\R_+^2$ is divided into two subsets $A$ and $B$ which are shown in pink and yellow colors respectively. The closed, convex set $A$ corresponds to the set above the curve  $\{(\alpha(s),\beta(s)) : \, s \geq 2\}\cup \{(\beta(s),\alpha(s)) : \, s \geq 2\}$ and $B = \R_+^2 \setminus A$. In Figure \ref{fig:convex}, $g$, which is the $1$-dimensional convexification of $f_0$ and its convex envelope $\tilde f$ are drawn in orange and green colors respectively. }
		\label{fig:domain-convex}
	\end{figure} 
	Then, the main goal of this section is to prove the following proposition:
	
	\begin{prop}\label{prop relax}
		
		Define
		\begin{align} \label{defn:convexification}
			f(a,b) =
			\begin{cases}
				\tilde{f} (a+b) \quad   &(a,b) \in A, \\
				f_0(a,b), \quad  &(a,b) \in B,
			\end{cases}  
		\end{align} where $f_0(a,b) := a \log a + b \log b + a b$ and $\tilde f (s) = f_0(\alpha (s), \beta (s)).$ Then, $\tilde f\in C^1 ([2,+\infty))$ and the function $f$ is the convex envelope of $f_0$.
	\end{prop}
	We give the proof of this proposition at the end of this section. Next, we give some technical results that will prove useful in the sequel. We define the ``product'' function $P$ as the following:
	\begin{align} \label{defn:Product_P}
		P(a,b) =
		\begin{cases}
			\pi (a+b) \quad &\mbox{ if }(a,b) \in A,\\
			ab \quad &\mbox{ if }(a,b) \in B,
		\end{cases}
	\end{align}
	where $\pi (s): = \alpha(s) \beta(s)\in C^0([2, +\infty))$.
	We gather some properties of $\pi(s)$ in the following lemma:
	\begin{lemma} \label{lem:Properties_pi} 
		The function $\pi \in C^0([2, +\infty))$, satisfies the following properties:
		\begin{enumerate}[label=\roman*.]
			\item \label{item1} $s-2\pi(s) >0$ for $s>2$. 
			\item \label{item2} $\pi\in C^1((2,+\infty))$ and $\pi'(s) =-\pi(s)\frac{(s-2)}{s-2\pi(s)}$.
			\item \label{item3} $\pi(s) <1$ for $s>2$.
			\item \label{item4} $\pi$ is also differentiable at $s=2$ and $\pi'(2) = -1/2$.
		\end{enumerate}
	\end{lemma}
	\begin{proof}
		Recall that $\alpha(s)$ and $\beta(s)$ are the minimizers of $g$, and that they satisfy $g''>0$. Since we have 
		$g'' (a) = \frac{1}{a} + \frac{1}{s-a} -2=\frac{s}{a(s-a)}-2$,
		taking $a=\alpha(s)$ and $s-a=\beta(s)$ we obtain $\frac{s}{\pi(s)}>2$, which proves \ref{item1}
		
		\medskip
		
		Now we would like to compute $\pi'(s)$ for $s>2$ (note that $\alpha$ and $\beta$ are differentiable because of the implicit function theorem): we have
		\begin{align} \label{eq:pi_prime}
			\pi'(s) = \alpha'(s) (s-\alpha(s)) + \alpha(s) (1-\alpha'(s)) = \alpha'(s) (s-2\alpha(s)) + \alpha(s).
		\end{align}
		Let us now compute $\alpha'(s)$ and $\beta'(s)$. We know from \eqref{eq:minima_f} that at the minima of $f$ we have
		\begin{align*}
			\frac{\alpha'(s)}{\alpha(s)} + (1-\alpha'(s)) = \frac{1-\alpha'(s)}{s-\alpha(s)} + \alpha'(s),
		\end{align*} which gives 
		\begin{align*}
			\alpha'(s) = \frac{\alpha(s) (1-s +\alpha(s))}{s-2\pi(s)} \quad \mbox{and} \quad \beta'(s) = \frac{\beta(s) (1-s +\beta(s))}{s-2\pi(s)}.
		\end{align*} Plugging these in \eqref{eq:pi_prime} we obtain
		\begin{align*} 
			\pi'(s) &= \frac{\alpha(s) (1-s +\alpha(s))}{s-2\pi(s)} (s-2\alpha(s)) + \alpha(s)
			= \frac{\alpha(s) (1-\beta(s))}{s-2\pi(s)} (\beta(s)-\alpha(s)) + \alpha(s) 
			\\&= \alpha(s) \left(1 +  \frac{(1-\beta(s)) (\beta(s)-\alpha(s)) }{s-2\pi(s)}\right) = - \pi(s) \frac{(s-2)}{s-2\pi(s)},
		\end{align*} which proves \ref{item2}
		
		\medskip
		
		Since $\pi'(s) \leq 0$ for $s>2$ (we use here $s-2\pi(s)>0$), and $\pi \in C^0([2, +\infty))$, we obtain that $\pi$ has its maximum value at $s=2$, then $\pi(s) < \pi(2) = 1$ for $s>2$. This gives \ref{item3}
		
		\medskip
		
		Now, we want to prove that $\pi$ is differentiable at $s=2$ and compute its derivative. We will consider the liminf and the limsup of the incremental ratio and bound it iteratively. We recall that we have $0 \leq \pi(s) \leq 1$, $-1 <  \pi '(s) \leq 0 $, $\pi \in C^0 ([2, + \infty))$, 
		and 
		\begin{align*}
			\pi'(s) = - \pi(s) \frac{1}{1- 2\frac{\pi(s)-1}{s-2}}. 
		\end{align*}
		We first note that we have  
		$$
		-1 < \frac{\pi(s)-1}{s- 2}  \leq 0.
		$$
		We then deduce
		\begin{align*}
			-1 \leq \liminf_{s\to 2^+}\frac{\pi(s)-1}{s- 2} \leq \limsup_{s\to 2^{+}} \frac{\pi(s)-1}{s- 2}  \leq 0.
		\end{align*}
		We define two sequences $(p_n)_{n \in \N}$ and $(q_n)_{n \in \N}$ which are meant to satisfy 
		\begin{align}\label{pnqn}
			p_n \leq \liminf_{s\to 2^+} \frac{\pi(s)-1}{s- 2} \leq \limsup_{s\to 2^{+}} \frac{\pi(s)-1}{s- 2}  \leq q_n.
		\end{align} We take $p_0 = -1$ and $q_0 = 0$. 
		Supposing that we have defined $p_n$ and $q_n$, we then note that for any $\varepsilon >0$, we have, for $s$ in a neighborhood of $2^+$, that
		\begin{align*}
			p_n  -\varepsilon  < \frac{\pi(s)-1}{s- 2}  < q_n +\varepsilon.
		\end{align*} This implies that we have, in the same neighborhood
		$$
		\pi'(s) < \frac{-\pi(s)}{1-2(p_n -\varepsilon)}.
		$$
		Since $\pi(s)\to 1$ as $s\to 2$, we can define $q_{n+1}$ via 
		$$q_{n+1} := \frac{-1}{1-2p_n }, $$ 
		and, analogously, 
		$$ 
		p_{n+1} := \frac{-1}{1-2q_n }.
		$$
		In particular, we have $p_1=-1$ and $q_1=-1/3$. We can see that the new values $p_{n+1}$ and $q_{n+1}$ also satisfy \eqref{pnqn}.
		From the definition of $p_n$ and $q_n$ we obtain
		\begin{align*}
			p_{n+2} = \frac{-1}{1-2 q_{n+1} } = - \frac{1-2p_n}{3-2 p_n}, \quad \text{and} \quad q_{n+2} = - \frac{1-2 q_n}{3-2 q_n}.
		\end{align*}
		By induction, we can see that the sequence $p_{2n}$ is increasing and bounded above by $-1/2$. If we denote its limit by $L$, we have $L=-\frac{1-2L}{2-2L}$, which implies $L = -1/2$. 
		The same holds for $p_{2n+1}=p_{2n}$.
		
		Similarly, $q_{2n}$ and $q_{2n+1}$ are decreasing and bounded from below by $-1/2$ and they converge to $-1/2$ as well.
		
		This gives $ \liminf_{s\to 2^+} \frac{\pi(s)-2}{s-2} =  -1/2$ and \ref{item4} is proven.
	\end{proof} 
	
	\begin{lemma} \label{lem:tilde_f_convex}
		The function	$\tilde{f}: [2,+\infty) \to \R$ (defined in Proposition \ref{prop relax}) is convex and $C^1$.
	\end{lemma} 
	\begin{proof} We recall that we have
		\begin{align*}
			\tilde{f}(s) := f_0 (\alpha(s), \beta(s)) = \alpha(s) \log (\alpha(s))+  \beta(s) \log (\beta(s)) + \alpha(s) \beta(s).
		\end{align*} Then we compute
		\begin{align*}
			\tilde{f}' (s)  
			= \left(  \log (\alpha (s))  + 1  + \beta (s)\right)  (\alpha ' (s) + \beta' (s))  =    \log (\alpha (s))  + 1  + \beta (s).
		\end{align*} 
		This allows to see $\tilde f\in C^1$ since the expression for $\tilde f'$ is made of continuous functions (as we do have $\alpha,\beta\in C^0([2,\infty))$). Moreover, we can go on differentiating and get
		\begin{align*}
			\tilde{f}''(s)   &= \frac{\alpha'(s)}{\alpha(s)} + \beta'(s) = \alpha'(s) + \frac{\beta'(s)}{\beta(s)} = 
			\frac{\alpha(s) (1-s +\alpha(s))}{s-2\pi(s)}  + \frac{ (1-s +\beta(s))}{s-2\pi(s)} 
			\\&= \frac{\alpha(s)-s \alpha(s)+\alpha^2(s) + 1-\alpha(s)}{s-2\pi(s)}  = \frac{1+ \alpha(s) (\alpha(s)-s)}{s-2\pi(s)} = \frac{1-\pi(s)}{s-2\pi(s)}.
		\end{align*} Since we have $\pi(s)\leq 1$ and $s-2\pi(s)\geq 0$, the second derivative of $\tilde f$ is non-negative, and $\tilde f $ is convex.
	\end{proof}

	\begin{remark} \label{rem:r_0}
		For future use, we denote by $r_0$ the number given by
		$$r_0:=\inf_{s>2}s\tilde f''(s)= \inf_{s>2} s\frac{1-\pi(s)}{s-2\pi(s)},$$
		and we note that we have $r_0>0$ since the function in the infimum is strictly positive, tends to $1$ as $s\to\infty$, and tends to $2/(2-1/\pi'(2))=1/2>0$ as $s\to 2^+$.
	\end{remark}
	\begin{corollary} \label{lem:f_convex}
		The function $f: \R_+^2 \to \R$ is convex.
	\end{corollary}
	\begin{proof} We notice that $f_0 $ is $C^1$ in the interior of $B$ and that $\tilde f$ is  $C^1$ on $[2,+\infty)$. Moreover we have the following formula for the gradient of $f$, using $s=a+b$:
		\begin{align*}
			\nabla f(a,b) = 
			\begin{cases}
				(\tilde f'(s), \tilde f'(s))=( \log (\alpha (s))  + 1  + \beta(s),\log (\beta (s))  + 1  + \alpha(s))  \quad &\mbox{ if }(a,b) \in A, \\
				\nabla f_0(a, b) = (
				\log a+ b + 1, \log b + a + 1 ), 
				\quad &\mbox{ if }(a,b) \in B.
			\end{cases}
		\end{align*}
		Since these two expressions agree on $\overline{B}\cap A$, then $f$ is globally $C^1$ in $(0,+\infty)^2$. This allows us to prove that $f$ is convex by considering separately its restrictions to $A$ and $B$.
		
		Indeed, convexity for $C^1$ functions is equivalent to the inequality $(\nabla f(x)-\nabla f(y))\cdot (x-y)\geq 0$ for every $x,y$. If the segment connecting $x$ and $y$ is completely contained either in $A$ or in $\overline{B}$ then the convexity of the two restrictions is enough to obtain the desired inequality. If not, we can decompose it into a finite number of segments (three at most) of the form $[x_i,x_{i+1}]$ with $x_0=x$ and $x_3=y$ and each $[x_i,x_{i+1}]$ fully contained either in $A$ or in $\overline{B}$. We then write
		$$(\nabla f(x)-\nabla f(y))\cdot (x-y)=\sum_{i=1}^3(\nabla f(x_i)-\nabla f(x_{i+1}))\cdot (x-y),$$ and the fact that $x-y$ is a positive scalar multiple of each vector $x_i-x_{i+1}$ shows that the convexity of each restriction is again enough for the desired result (note that we strongly use here $f\in C^1$, i.e. that the gradients of the two restrictions agree).

		The convexity of $f$ restricted to $\overline{B}$ comes from the positivity of the Hessian of $f_0$ and that of $f$ restricted to $A$ from the convexity of $\tilde f$, and the result is proven.
	\end{proof}
	
	Now, with the help of the above results, we prove Proposition \ref{prop relax}.
	\begin{proof}[Proof of Proposition \ref{prop relax}]
		The function $f$ has been built so that on each segment $\{(a,b):a+b=s\}$ it coincides with the convexification of the restriction of $f_0$ on such a segment. So, the convexification of $f_0$ cannot be larger than $f$. On the other hand, $f$ is a convex function smaller than $f_0$, so it is also smaller than the convexification, which proves the claim.  
	\end{proof}
	
	We conclude this section with a remark which will be useful in the sequel (see Lemma \ref{lem:de_Giorgi_H^1} in Section~\ref{sec:EDI}).
	
	\begin{remark} \label{rem:local_Lipschitz}
		If $\chi$ is a function which is compactly supported in the set $B$, then there exists a Lipschitz continuous function $g:\R^2\to\R$ such that we have
		$$
		\chi(a,b) = g(f_a(a,b),f_b(a,b)).
		$$ 
		This holds because the Jacobian of $(f_a,f_b)$ is invertible inside $B$. 
	\end{remark}
	
	\section{Lower semi-continuity of the slope}\label{sec l.s.c.}
	
	The goal of this section is to give a precise definition of the slope functional $\Slo(\rho,\mu)$ (introduced in \eqref{roughdefislo}) and prove that it is lower semi-continuous with respect to the weak topology on measures.

	Notice that the formula \eqref{roughdefislo} for $\Slo(\rho,\mu)$ makes use of the gradients of $\rho$ and $\mu$, but this expression is not well-defined for any arbitrary couple of measures $(\rho,\mu)$. This leads us to consider the following space:
	\begin{definition} \label{defn: H}
		We define the class $\mc H$ as the set of all pairs of densities $\rho, \mu \in L^1(\O)\cap \mc P(\O) $ such that 
		\begin{enumerate}[label=\roman*.]
			\item \label{H:prop1} For every $\eta \in W_c^{1,\infty}(B)$, we have $\eta(\rho,\mu) \in H^1(\O)$;
			\item \label{H:prop2} We also have $\sqrt{\rho+\mu} \in H^1(\O)$.
		\end{enumerate}
	\end{definition}
	
	The sets $A,B$ above are those defined in Section \ref{sec:convexification} (and we keep this notation in the whole presentation), and by  $W_c^{1,\infty}(B)$, we denote the set of Lipschitz functions whose support is compact inside $B$ (it can touch the axes $\R\times \{0\}$ and $\{0\}\times \R$, but not the separating curve between $A$ and $B$).
	
	\medskip
	For $(\rho,\mu)\in \mc H$, we do not have necessarily $\rho, \mu \in H^1(\O)$. However, we can define for couples $(\rho,\mu)\in \mc H$ a suitable notion of ``gradient'' for certain functions of $(\rho,\mu)$.  The notion of gradient we want to define should satisfy at least some chain-rule in order to be useful in the sequel, that is, we would like to have, for any $(\rho,\mu)\in \mc H$, that 
	$$
	\nabla (\chi(\rho,\mu)) = \partial_a \chi(\rho,\mu)\nabla \rho + \partial_b\chi(\rho,\mu)\nabla \mu.
	$$
	In particular, we need this to be true for some simple functions $\chi$, such as affine functions composed with suitable positive parts so that we have $\supp(\chi)\subset B$.
	
	Hence, let us define, for $(\alpha,\beta,c)\in \R_+^3$, the following function:
	\begin{equation}\label{T}
		T_{\alpha,\beta,c}(a,b) = \max\{c-\alpha a - \beta b , 0\}, 
	\end{equation}
	which we will refer to as ``triangle'' sometimes.
	Owing to the convexity of the set $A$, for every $(a,b)\in B$ we can find $(\alpha,\beta,c)\in \R_+^3$ such that $T_{\alpha,\beta,c}(a,b)>0$ and $T_{\alpha,\beta,c}$ is compactly supported in $B$. Whenever $\supp T_{\alpha,\beta,c}\subset B$, then $T_{\alpha,\beta,c}(\rho,\mu)\in H^1(\Omega)$ for every $(\rho,\mu)\in \mc H$. 
	
	\begin{definition}
		Let us fix a countable dense set $E$ of parameters $(\alpha,\beta,c)$: for simplicity we choose $E=\mathbb Q_+^3$. Given $(\rho,\mu)\in \mc H$, for each $(\alpha,\beta,c)\in E$ fix a representative of the weak gradient of $T_{\alpha,\beta,c}(\rho,\mu)$. Take $x\in \Omega$ such that $(\rho(x),\mu(x)) \in B$ and $(\alpha,\beta,c) \in E$ and $\e>0,\e\in\mathbb Q$ such that the supports of $T_{\alpha,\beta,c}, T_{\alpha-\e,\beta,c},T_{\alpha,\beta-\e,c} $ are contained in $B$ and $(\rho(x),\mu(x)) \in \supp T_{\alpha,\beta,c}$. Then we define the gradients of $\rho (x)$ and $\mu(x)$ as the following:
		\begin{multline*}
			(\nabla \rho(x),\nabla \mu(x)) := 
			\Big (\nabla \Big (\frac{1}{\e}(T_{\alpha-\e,\beta,c}(\rho,\mu) - T_{\alpha,\beta,c}(\rho,\mu)) \Big )(x),\nabla \Big (\frac{1}{\e}(T_{\alpha,\beta-\e,c}(\rho,\mu) - T_{\alpha,\beta,c}(\rho,\mu)) \Big )(x)\Big).
		\end{multline*}
	\end{definition} 
	
	Let us observe that the above definition of the gradients does not depend on the choices of $\alpha,\beta,c$, and $\e$, except possibly on a negligible set of points $x$. Indeed, for a.e. $x$ such that $(\rho(x),\mu(x)) \in \supp T_{\alpha,\beta,c}\cap T_{\tilde \alpha, \tilde \beta, \tilde c}$, then the functions $(T_{\alpha-\e,\beta,c}(\rho,\mu) - T_{\alpha,\beta,c}(\rho,\mu)) / \e$ and $(T_{\tilde \alpha -\e,\tilde \beta ,\tilde c}(\rho,\mu) - T_{\tilde \alpha,\tilde \beta,\tilde c}(\rho,\mu)) / \e$ are in $H^1$ and coincide, hence their gradients coincide almost everywhere at these points, see \cite{B}.
	
	\medskip 
	Let us also mention that, with such definition of the gradients, we have the following chain-rule: for any $\chi \in W_c^{1,\infty}(B)$, and for any $(\rho,\mu)\in \mc H$ we have
	\begin{align} \label{chain_rule}
		\chi(\rho,\eta)\in H^1(\O) \quad  \mbox{and} \quad  \nabla (\chi(\rho,\mu)) = \chi_a(\rho,\mu)\nabla \rho + \chi_b(\rho,\mu)\nabla \mu.
	\end{align}
	
	To prove this fact, let us start with considering the case where $\chi$ is compactly supported in $\supp T_{\alpha,\beta,c}$, for some $\alpha, \beta ,c\in \mathbb Q$. Then, for $\e$ small enough, we have, for $x$ such that $(\rho(x),\mu(x))\in \supp \chi$,
	\begin{align*}
		\chi(\rho,\mu) = \chi\Big(\frac{1}{\e}(T_{\alpha-\e,\beta,c}(\rho,\mu) - T_{\alpha,\beta,c}(\rho,\mu)),\frac{1}{\e}(T_{\alpha,\beta-\e,c}(\rho,\mu) - T_{\alpha,\beta,c}(\rho,\mu)) \Big).
	\end{align*}
	
	As both the functions are in $H^1$, their gradients coincide on the set where $(\rho,\mu)\in \supp \chi$, and on this set, applying the chain rule on the right-hand side yields the desired equality. Outside of this set, the chain rule is direct and everything is zero. We then conclude by observing that any compact subset of $B$ is contained in a finite union of supports of sets of the form $\supp T_{\alpha,\beta,c}$ for $\alpha, \beta ,c\in \mathbb Q$, and then we use a partition of the unity to generalize the result to general $\chi \in W_c^{1,\infty}(B)$.
	
	We note that, if $\rho,\mu \in H^1(\O)$, then the gradient defined above coincide with the usual gradient.
	
	\medskip
	
	To conclude these remarks on the gradients of $(\rho,\mu)$, we observe that our definition of the space $\mc H$ does not allow us to define the gradients of $\rho,\mu$ properly when $(\rho,\mu)$ lies in the set $A$. However, in this set, and this will be sufficient for our needs, the gradient of the sum $\rho + \mu$ is well defined, i.e, it is measurable. Indeed, we know that $\sqrt{\rho+\mu}$ is in $H^1$, hence, we set $\nabla (\rho+\mu) := 2(\sqrt{\rho+\mu})\nabla \sqrt{\rho+\mu}$.
	
	It follows from the above discussion that, for any Lipschitz function $h=h(a,b)$ that depends only on $a+b$, if $(a,b)\in A$, then $\nabla (h(a,b))$ is well defined.
	
	\medskip
	
	We now define the slope functional on the space $\mc H$.
	\begin{definition}\label{def slope}
		Let $(\rho,\mu)\in \mc H$. Then, we define the slope functional $\Slo$ as
		$$
		\Slo(\rho,\mu) := \int_{(\rho,\mu)\in B} \left \vert\frac{\nabla \rho}{\rho} + \mu \right \vert^2\rho + \left \vert \frac{\nabla \mu}{\mu} + \rho \right \vert^2\mu + \int_{(\rho,\mu)\in A} (\tilde f^{\prime\prime})^2(S) \vert\nabla S\vert^2S,
		$$
		where $S:=\rho+\mu$. 
	\end{definition}
	Note that the above formula for the slope has been obtained by expanding the expression
	$$
	\Slo(\rho,\mu) := \int\vert\nabla f_a(\rho,\mu) \vert^2\rho + \int\vert\nabla f_b(\rho,\mu) \vert^2\mu.
	$$
	
	We are now in a position to state the main result that we prove in this section.
	
	\begin{theorem}\label{th l.s.c.}
		Let $(\rho_n,\mu_n) \in \mc H$ be such that
		$$
		(\rho_n,\mu_n)\underset{n\to+\infty}{\longrightarrow} (\rho,\mu) \in \mc H,
		$$
		where the above convergence is weak in the sense of measures. Then we have
		$$
		\Slo(\rho,\mu)\leq \liminf_{n\to +\infty} \Slo(\rho_n,\mu_n).
		$$
	\end{theorem}
	
	\subsection{Preliminary results}
	\label{sec:preliminaries_l.s.c.}
	In this section, we provide some preliminary results that will be used in the proof of Theorem \ref{th l.s.c.}.
	We start with proving the following proposition:
	\begin{prop}\label{prop cv B}
		Assume that $(\rho_n,\mu_n)\in \mc H$ converges weakly, as $n \to +\infty$, to $(\rho,\mu)\in \mc P(\O)$, and that $\Slo(\rho_n,\mu_n)$ is bounded independently of $n$. Then, for any Lipschitz continuous function $\chi(a,b)$ which is constant everywhere on $A$, we have
		$$
		\chi(\rho_n,\mu_n) \underset{n\to +\infty}{\longrightarrow} \chi(\rho,\mu) \quad \mbox{ strongly in } L^2.
		$$
	\end{prop}

	The proof of Proposition \ref{prop cv B} relies on two lemmas.
	
	\begin{lemma}\label{lem cv B}
		Assume that $(\rho_n,\mu_n)$ and $\chi$ satisfy the hypotheses of Proposition \ref{prop cv B}. Assume in addition that $\chi$ is constant outside a compact set contained in $B$. Then, the sequence $(\chi(\rho_n,\mu_n))_{n\in\N}$ is bounded in $H^1(\O)$. In particular, it has a subsequence that converges strongly in  $L^2$ and weakly in $H^1$.
	\end{lemma}

	\begin{proof} We start by defining, for all points $x$ such that $(\rho_n(x),\mu_n(x))\in B$, the following vector fields:
		\begin{equation}\label{def XY}
			X_n := \frac{\nabla \rho_n}{\rho_n} + \nabla \mu_n, \quad Y_n : = \frac{\nabla \mu_n}{\mu_n} + \nabla \rho_n.
		\end{equation}
		Note that $X_n$ is only defined on $\{\rho_n>0\}$, i.e. $\rho_n-$a.e., and $Y_n$ on $\{\mu_n>0\}$, i.e. $\mu_n-$a.e.
		Therefore, by the definition of $\Slo$, for any function $\eta$ compactly supported in $B$, we have that
		$$
		\int_\O (\rho_n\vert X_n\vert^2 + \mu_n\vert Y_n\vert^2)\eta(\rho_n,\mu_n) 
		$$
		is bounded independently of $n$.
		
		Again, for points $x$ such that $(\rho_n(x),\mu_n(x))\in B$, we can write
		$$
		\nabla \rho_n = \frac{\rho_n X_n - \rho_n\mu_n Y_n}{1-\rho_n \mu_n}, \quad \nabla \mu_n = \frac{\mu_n Y_n - \rho_n\mu_n X_n}{1-\rho_n \mu_n}.
		$$
		Hence, remembering that $\rho_n\mu_n < 1$ for $(\rho_n,\mu_n)\in B$, we obtain
		$$
		\eta(\rho_n,\mu_n)\vert \nabla \rho_n\vert^2 \leq  2\eta(\rho_n,\mu_n)\frac{\rho_n^2 \vert X_n\vert^2 + (\rho_n\mu_n)^2 \vert Y_n\vert^2}{(1-\rho_n \mu_n)^2}\leq 2\eta(\rho_n,\mu_n) \rho_n\frac{\rho_n \vert X_n\vert^2 + \mu_n \vert Y_n\vert^2}{(1-\rho_n \mu_n)^2},
		$$
		and
		$$
		\eta(\rho_n,\mu_n)\vert \nabla \mu_n\vert^2 \leq   2\eta(\rho_n,\mu_n) \mu_n\frac{\rho_n \vert X_n\vert^2 + \mu_n \vert Y_n\vert^2}{(1-\rho_n \mu_n)^2}.
		$$
		
		\medskip 
		Since $(\rho_n,\mu_n) \in \mc H$, we have that $\chi(\rho_n,\mu_n) \in H^1(\O)$ (where $\chi$ is given as in the statement of Proposition \ref{prop cv B}). Up to subtracting a constant, we can assume that $\chi$ is compactly supported in $B$. Now, let us prove that the $H^1$ norm of $\chi(\rho_n,\mu_n)$ is bounded independently of $n$.
		\begin{align*}
			\int_\O \vert \nabla \chi(\rho_n,\mu_n)\vert^2 &\leq 2\int_\O \vert \chi_a(\rho_n,\mu_n)\vert^2\vert \nabla \rho_n\vert^2 + \vert \chi_b(\rho_n,\mu_n)\vert^2\vert \nabla \mu_n\vert^2 \\
			&\leq 4\int_\O \frac{\vert \chi_a(\rho_n,\mu_n)\vert^2\rho_n}{(1-\rho_n\mu_n)^2}(\rho_n \vert X_n\vert^2 + \mu_n \vert Y_n\vert^2) + \frac{\chi_b(\rho_n,\mu_n)\vert^2\mu_n}{(1-\rho_n\mu_n)^2}(\mu_n \vert Y_n\vert^2+ \rho_n \vert X_n\vert^2).
		\end{align*}
		The first line above is obtained by using the chain rule given by \eqref{chain_rule} for the composition of Lipschitz functions and functions in $\mc H$.
		
		The quantities $\vert \chi_a(\rho_n,\mu_n)\vert^2\rho_n/ (1-\rho_n\mu_n)^2$ and $\vert \chi_b(\rho_n,\mu_n)\vert^2\mu_n/ (1-\rho_n\mu_n)^2$ are bounded because the support of $\chi$ is far from the set $A$, so that the product $\rho_n\mu_n$ is bounded above by a constant strictly less than $1$ on the set of points $x$ such that $(\rho_n(x),\mu_n(x))\in\supp(\chi)$. Hence the result follows.
	\end{proof}
	
	Let us now improve the above lemma by showing that, for $\chi$ as above, the $H^1$ weak and $L^2$ strong limit of $\chi(\rho_n,\mu_n)$ is $\chi(\rho,\mu)$, i.e. we can pass to the limit inside the function $\chi$. To do so, we start with considering the specific case were $\chi$ is of the form $T_{\alpha,\beta,c}$, defined in \eqref{T}. 
	\begin{lemma}\label{lemma affine}
		Let $(\alpha,\beta,c) \in \R_+^3$ be such that $\supp T_{\alpha,\beta,c}\subset B$. Then
		$$
		T_{\alpha,\beta,c}(\rho_n,\mu_n) \underset{n\to+\infty}{\longrightarrow} T_{\alpha,\beta,c}(\rho,\mu),
		$$
		and the convergence is strong in $L^2$ and weak in $H^1$.
	\end{lemma}
	
	\begin{proof}
		We assume that $(\alpha,\beta,c)$ are chosen as in the statement of the lemma, and we omit writing them as subscripts of $T_{\alpha,\beta,c}$ in the proof.
		
		By Lemma \ref{lem cv B}, there exists $u\in H^1$, $u\geq 0$, such that
		$$
		T(\rho_n,\mu_n)\to u,
		$$
		strongly in $L^2$ and weakly in $H^1$. Using
		$$
		c - \alpha\rho_n-\beta\mu_n  \leq T(\rho_n,\mu_n),
		$$
		and the weak converge of $c - \alpha\rho_n-\beta\mu_n$ to $c-\alpha\rho-\beta\mu$, we find that $c - \alpha\rho-\beta\mu \leq u$. Taking the maximum with $0$, we get
		\begin{align}
			\label{u}
			T(\rho,\mu)\leq u.
		\end{align} 
		This already proves the equality $ T(\rho,\mu)=u$ on $\{u=0\}$. 
		
		Now, let $\delta>0$ be fixed and define the set $\omega := \{ u>\delta \} \subset \O$. Let $\e>0$ be fixed as well. Using Egoroff's theorem, we can find $E\subset \omega$ such that $\vert E\vert <\e$ and such that $T(\rho_n,\mu_n)$ converges uniformly to $u$ on $\omega\backslash E$. Taking $n$ large enough, we have 
		$$
		T(\rho_n,\mu_n)\mathbbm{1}_{\omega\backslash E} = (c - \alpha \rho_n - \beta \mu_n )\mathbbm{1}_{\omega\backslash E}.
		$$
		The term on the left-hand side converges to $u \mathbbm{1}_{\omega\backslash E}$ strongly in $L^2$ and the term on the right-hand side converges to $(c-\alpha\rho - \beta\mu )\mathbbm{1}_{\omega\backslash E}$ weakly. Then,
		$$
		u = c-\alpha\rho - \beta\mu  \leq T(\rho,\mu) \quad \text{ on } \omega\backslash E,
		$$
		and this is actually an equality due to \eqref{u}. Since the measure of $E$ can be taken arbitrarily small, and up to taking $\delta \to 0$,  we obtain that $u=T(\rho,\mu)$ a.e.
	\end{proof}
	
	We can now turn to the proof of Proposition \ref{prop cv B}.
	\begin{proof}[Proof of Proposition \ref{prop cv B}.]
		If the function $\chi$ is of the form $T_{\alpha,\beta,c}$, then Lemma \ref{lemma affine} tells us that the proposition is true. The strategy of the proof is to prove first that the proposition holds true for any function $\chi$ whose support is contained in a triangle, itself contained in $B$, then we show that it holds true for any function $\chi$ whose support does not touch $A$ (but may not be contained in a single triangle), and finally we consider the general case.
		
		\medskip 
		
		\textit{ \bf Step 1. The case where the support of $\chi$ is in a triangle.} Assume that the support of $\chi$ is compactly supported in the support of the triangle function $T_{\alpha+\ve,\beta+\ve,c-\ve}$ and that $(\alpha,\beta,c)$ and $\ve>0$ are chosen so that such a support is contained in $B$. We now define $T_1:=T_{\alpha+\ve,\beta,c}$ and $T_2:=T_{\alpha,\beta+\ve,c}$. We want to prove that $\chi(\rho,\mu)$ can be expressed as a Lipschitz function of $(T_1(\rho,\mu),T_2(\rho,\mu))$. To do so, we observe that there exists an affine function $L:\R^2\to\R^2$ such that $(a,b)=L(c-(\alpha+\ve)a-\beta b,c-\alpha a-(\beta+\ve)b)$.
		We then define a function $g:L^{-1}(\R_+^2)\cap (\R_+^2)\to\R$ via the formula
		$$g(t_1,t_2):=\begin{cases} 0 &\mbox{ if }t_1t_2=0,\\
			\chi(L(t_1,t_2))&\mbox{ if }t_1t_2>0.
		\end{cases}
		$$
		It is clear that $\chi(\rho,\mu)$ equals to $g(T_1(\rho,\mu),T_2(\rho,\mu))$ since if either $T_1(\rho,\mu)$ or $T_2(\rho,\mu)$ vanishes, then we have $\chi(\rho,\mu)=0$, while in the other case we can express $\rho$ and $\mu$ via the affine function $L$, and of course we only need to apply $g$ to values which are in $\R_+^2$ (since $T_1,T_2\geq 0$) and in $L^{-1}(\R_+^2)$ (since $\rho,\mu\geq 0$). We only need to prove that $g$ is Lipschitz continuous, which is not evident from its definition. To do so, we will prove that we have $g(t_1,t_2)=0$ if $t_1<\ve$ or $t_2<\ve$. 
		
		\medskip
		Indeed, setting $(a,b)=L(t_1,t_2)$, the condition $g(t_1,t_2)>0$ implies that $t_1,t_2>0$ and $(\alpha+\ve)a+(\beta+\ve)b+\ve<c$, i.e. $\ve b+\ve<t_1$, hence $t_1>\ve$ since $b>0$. Analogously, we also have $\ve a +\ve<t_2$, hence $t_2>\ve$. This shows that we could define 
		$$g(t_1,t_2):=\begin{cases} 0 &\mbox{ if }t_1<\ve \mbox{ or }t_2<\ve,\\
			\chi(L(t_1,t_2))&\mbox{ if }t_1t_2>0,
		\end{cases}
		$$
		and both expressions are Lipschitz continuous and they agree on the open set which is the intersection of the two domains of definition.
		
		Once we know that $\chi(\rho,\mu)$ can be written as $g(T_1(\rho,\mu),T_2(\rho,\mu))$, the claim follows.

		\medskip
		\textit{\bf Step 2. The case where the support of $\chi$ does not touch $A$.} Assume that the support of $\chi$ is compactly supported in $B$. We use the fact, which is based on the convexity of $A$, that the domain $B$ is a union of triangles of the form $\supp T$, even if functions supported in $B$ are not necessarily supported in one of such of triangles only. Hence, we can find a finite family $((\alpha_k,\beta_k,c_k) )_{k}$ such that
		
		$$
		\supp \chi \subset \bigcup_{k} \supp T_{\alpha_k,\beta_k,c_k}.$$
		
		Let $(\chi_k)$ be a family of functions compactly supported in $\supp T_{\alpha_k,\beta_k,c_k}$, such that
		$\chi= \sum_k \chi_k.$
		Then, we can apply the first step to conclude.
		
		\medskip
		\textit{ \bf Step 3. The general case.} 
		Up to subtracting a constant, we assume that $\chi=0$ on $A$. Let us start with assuming that $\chi$ is non-negative. We proceed by approximation. Let $\e>0$ be fixed, and define
		$$
		\chi_\e(a,b) := [\chi(a,b)-\e]_+.
		$$
		By Step 2 above, we have that
		$$
		\chi_\e(\rho_n,\mu_n) \to \chi_\e(\rho,\mu),
		$$
		and this convergence (up to a subsequence) holds strongly in $L^2$ and weakly in $H^1$. Then,
		\begin{align*}
			\vert \chi(\rho_n,\mu_n) - \chi(\rho,\mu)\vert &\leq \vert \chi(\rho_n,\mu_n) -  \chi_\e(\rho_n,\mu_n)\vert +   
			\vert \chi_\e(\rho_n,\mu_n) - \chi_\e(\rho,\mu)\vert
			+\vert \chi_\e(\rho,\mu) -  \chi(\rho,\mu)\vert \\
			&\leq 2\e + \vert \chi_\e(\rho_n,\mu_n) - \chi_\e(\rho,\mu)\vert.
		\end{align*}
		Taking the limit as $n \to +\infty$ yields the result.

		To treat the case where $\chi$ changes its sign, we can apply the argument on the positive and negative parts of $\chi$ separately.
	\end{proof}

	We conclude these preliminary results with the following lemma:
	
	\begin{lemma}\label{cv grad weak}
		Let $(\rho_n,\mu_n) \in \mc H$ converges weakly, as $n \to +\infty$, to $(\rho,\mu)\in \mc H$ and be such that $\Slo(\rho_n,\mu_n)$ is bounded independently of $n$. Then, for any $\chi$ compactly supported in $B$, we have
		$$
		\chi(\rho_n,\mu_n)\nabla\rho_n\deb \chi(\rho,\mu)\nabla\rho\quad \mbox{and}\quad 
		\chi(\rho_n,\mu_n)\nabla\mu_n\deb \chi(\rho,\mu)\nabla\mu,
		$$
		where the convergence is weak in $L^2$.
	\end{lemma}
	\begin{proof}
		As usual, by possibly decomposing $\chi$ into a finite sum we can assume that the support of $\chi$ is included in a triangle. We choose two different triangle functions $T_1:=T_{\alpha+\ve,\beta,c},T_2:=T_{\alpha,\beta,c}$ such that their supports include that of $\chi$. We use the weak $H^1$ convergence of $T_2(\rho_n,\mu_n)$ to $T_2(\rho,\mu)$ to deduce
		$$\ind_{\supp T_2}(\rho_n,\mu_n)(\alpha\nabla\rho_n+\beta\nabla\mu_n)\deb \ind_{\supp T_2}(\rho,\mu)(\alpha\nabla\rho+\beta\nabla\mu).$$
		This implies 
		$$\chi(\rho_n,\mu_n)(\alpha\nabla\rho_n+\beta\nabla\mu_n)\deb \chi(\rho,\mu)(\alpha\nabla\rho+\beta\nabla\mu)$$
		since we just need to multiply the above weak converging sequence with $\chi(\rho_n,\mu_n)$, which is dominated and pointwisely converging a.e. to $\chi(\rho,\mu)$.
		If we do the same for $T_1$ we also obtain
		$$\chi(\rho_n,\mu_n)((\alpha+\ve)\nabla\rho_n+\beta\nabla\mu_n)\deb \chi(\rho,\mu)((\alpha+\ve)\nabla\rho+\beta\nabla\mu).$$
		Subtracting the two relations and dividing by $\ve$, we obtain
		$$\chi(\rho_n,\mu_n)\nabla\rho_n\deb \chi(\rho,\mu)\nabla\rho.$$
		We can now also deduce that $\chi(\rho_n,\mu_n)\nabla\mu_n\deb \chi(\rho,\mu)\nabla\mu$ and the claim is proven.
	\end{proof}
	
	We make use of these preliminary results to prove Theorem \ref{th l.s.c.}.

	\subsection{Lower semi-continuity in the region \texorpdfstring{$B$}{B}}
	
	The goal of this section is to prove the following:
	
	\begin{prop}\label{th l.s.c. B}
		Let $(\rho_n,\mu_n) \in \mc H$ be such that $(\rho_n,\mu_n)$ converges weakly, as $n \to +\infty$, to $(\rho,\mu)\in\mc H$ and such that $\Slo(\rho_n,\mu_n)$ is bounded independently of $n$. Let $\chi$ be a Lipschitz function compactly supported in $B$. Then
		\begin{multline*}
			\int_\O \left(\vert \nabla f_a(\rho,\mu)\vert^2 \rho + \vert \nabla f_b(\rho,\mu)\vert^2 \mu\right)\chi(\rho,\mu) \\ \leq \liminf_{n\to +\infty} \int_\O \left(\vert \nabla f_a(\rho_n,\mu_n)\vert^2 \rho_n + \vert \nabla f_b(\rho_n,\mu_n)\vert^2 \mu_n\right)\chi(\rho_n,\mu_n).
		\end{multline*}
	\end{prop}
	
	\begin{proof}
		By the definitions of $f$ and $\chi$, we have
		\begin{align*}
			\int_\O (\vert \nabla( f_a( \rho_n,\mu_n))\vert^2\rho_n &+ \vert \nabla( f_b( \rho_n,\mu_n))\vert^2\mu_n)\chi(\rho_n,\mu_n) \\
			&= \int_\O \bigg (\left \vert \frac{\nabla \rho_n}{\rho_n} + \nabla \mu_n \right \vert^2\rho_n + \left \vert \frac{\nabla \mu_n}{\mu_n} + \nabla \rho_n \right \vert^2\mu_n \bigg )\chi(\rho_n,\mu_n)\\
			&=\int_\O \left( \frac{\vert \nabla \rho_n+\rho_n \nabla \mu_n \vert^2}{\rho_n} +  \frac{\vert \nabla \mu_n + \mu_n \nabla \rho_n\vert^2}{\mu_n}\right)\chi(\rho_n,\mu_n)\\
			&= \int_\O \left( \frac{\vert (\nabla \rho_n+\rho_n \nabla \mu_n)\chi(\rho_n,\mu_n) \vert^2}{\rho_n \chi(\rho_n,\mu_n)} +  \frac{\vert \nabla (\mu_n + \mu_n \nabla \rho_n )\chi(\rho_n,\mu_n)\vert^2}{\mu_n \chi(\rho_n,\mu_n)}\right).
		\end{align*}
		
		The latter expression is convex in the terms $(\nabla \rho_n + \rho_n\nabla \mu_n)\chi
		(\rho_n,\mu_n)$, $\rho_n\chi(\rho_n,\mu_n)$ (and same for the terms involving $\mu_n$). Then the standard lower semi-continuity results (see Chapter 4 in \cite{Giusti}) will prove the claim. We only have to note that Proposition \ref{prop cv B} provides
		$$
		\rho_n\chi
		(\rho_n,\mu_n)\underset{n\to+\infty}{\longrightarrow} \rho \chi
		(\rho,\mu),
		$$
		where the convergence is strong in $L^2$ (and the same also holds true for $\mu_n$) and Lemma \ref{cv grad weak} in turn provides
		$$
		(\nabla \rho_n + \rho_n\nabla \mu_n)\chi(\rho_n,\mu_n) \underset{n\to+\infty}{\longrightarrow}
		(\nabla \rho + \rho\nabla \mu)\chi(\rho,\mu),
		$$
		where the convergence is weak in $L^2$. We used here many times that the functions $(a,b)\mapsto a\chi(a,b),b\chi(a,b)$ are compactly supported in $B$ and Lipschitz continuous. This concludes the proof.
	\end{proof}

	\subsection{Lower semi-continuity in the region \texorpdfstring{$A$}{A}}

	The goal of this section is to prove the following:
	\begin{prop}\label{th l.s.c. 1}
		Let $(\rho_n,\mu_n) \in \mc H$ such that $(\rho_n,\mu_n)$ converges weakly, as $n \to +\infty$, to $(\rho,\mu)\in \mc H$ and such that $\Slo (\rho_n, \mu_n)$ is bounded independently of $n$. Let $\chi$ be a Lipschitz function compactly supported in $B$ such that $0\leq \chi \leq 1$. Then
		\begin{multline*}
			\int_{(\rho,\mu)\in A} \left(\vert \nabla f_a(\rho,\mu)\vert^2 \rho + \vert \nabla f_b(\rho,\mu)\vert^2 \mu\right)\\ \leq \liminf_{n\to +\infty} \int_\O \left(\vert \nabla (f_a(\rho_n,\mu_n)\vert^2 \rho_n + \vert \nabla (f_b(\rho_n,\mu_n)\vert^2 \mu_n\right)(1-\chi(\rho_n,\mu_n)).
		\end{multline*}
	\end{prop}
	
	Let us define the function $S := \rho + \mu$ and consider the function $P(\rho, \mu)$ defined in \eqref{defn:Product_P}. We recall that (see Lemma \ref{lem:Properties_pi}), for $s>2$, we have
	
	$$
	\tilde f^{\prime\prime}(s) = \frac{1-\pi(s)}{s-2\pi(s)}.
	$$
	
	Remember that, if $(\rho,\mu)\in \mc H$, then the gradients of $\sqrt{S}$ and $ P(\rho,\mu)$ are well defined. We can then state the following lemma:
	\begin{lemma} \label{lem: ineq}
		Consider $(\rho,\mu) \in \mc H$. Then we have
		$$
		\frac{\vert \nabla(S + P(\rho,\mu))\vert^2}{S}\leq
		\vert \nabla f_a(\rho,\mu)\vert^2 \rho + \vert \nabla f_b(\rho,\mu)\vert^2 \mu.
		$$
	\end{lemma}
	
	\begin{proof}
		We consider the cases where $(\rho,\mu)\in A$ and $(\rho,\mu)\in B$ separately. The inequality is actually an equality in the set $A$.
		
		We start with the set $B$. We have, by convexity, for $(\rho,\mu)\in B$,
		\begin{align*}
			\left (\vert \nabla f_a(\rho,\mu)\vert^2 \rho + \vert \nabla f_b(\rho,\mu)\vert^2 \mu\right) &=
			\bigg(\left|\frac{\nabla \rho}{\rho} + \nabla \mu \right|^2\rho + \left|\frac{\nabla \mu}{\mu} + \nabla \rho \right|^2 \mu\bigg) \\ 
			&\geq (\rho+\mu)\left|\frac{\rho}{\rho+\mu} \bigg (\frac{\nabla \rho}{\rho} + \nabla \mu \bigg ) + \frac{\mu}{\rho+\mu}  \bigg (\frac{\nabla \mu}{\mu} + \nabla \rho \bigg )  \right|^2\\
			&= \frac{1}{\rho + \mu}\left| \nabla(\rho +\mu) + \nabla (\rho \mu)  \right|^2 \\ &= \frac{\vert \nabla(S + P(\rho,\mu))\vert^2}{S}.
		\end{align*}
		Now, let us consider the set $A$. For  $(\rho,\mu)\in A$ we have 
		$$\vert \nabla f_a(\rho,\mu)\vert^2 \rho + \vert \nabla f_b(\rho,\mu)\vert^2 \mu= (\tilde f^{\prime\prime})^2(S) \vert\nabla S\vert^2S.
		$$
		We also have
		$$
		\nabla (S + \pi(S))= \nabla S (1 + \pi^\prime(S)).
		$$
		Owing to the relation $\pi^\prime = -\pi \frac{(s-2)}{s-2\pi}$, we obtain
		$$
		\nabla (S + \pi(S)) = \nabla S \frac{S-2\pi(S)-S\pi(S)+2\pi(S)}{S-2\pi(S)}= S \tilde f^{\prime \prime}(S)\nabla S.$$
		This proves that the desired inequality is actually an equality when $(\rho,\mu ) \in A$.
	\end{proof}

	Lemma \ref{lem: ineq} implies that, if $(\rho_n,\mu_n) \in \mc H$ is such that $\Slo(\rho_n,\mu_n)$ is bounded independently of $n$, then so is the $H^1$ norm of 
	$$
	\sqrt{S(\rho_n,\mu_n) + P(\rho_n,\mu_n)}\in H^1(\O).
	$$
	
	\begin{prop}
		Let $(\rho_n,\mu_n) \in \mc H$ be such that $\Slo(\rho_n,\mu_n)$ is bounded, and suppose that  $\rho_n\deb\rho$ and $\mu_n\deb\mu$ as $n \to +\infty$. Let $S_n := \rho_n+\mu_n$ and let $P_n := P(\rho_n,\mu_n)$. Then, we have the following convergence results:
		\begin{itemize}
			\item $S_n \to \rho +\mu$ a.e.
			\item $P_n \to P(\rho,\mu)$ a.e.
		\end{itemize}
	\end{prop}
	
	\begin{proof}
		We start by giving some bounds on $\nabla S_n$. Recalling the definitions of $X_n,Y_n$ from \eqref{def XY}, we  have (on the set  $\{(\rho_n,\mu_n)\in B\}$ of points $x\in\Omega$ where  $(\rho_n(x),\mu_n(x))\in B$),
		$$
		\nabla \rho_n = \frac{\rho_n X_n - \rho_n\mu_n Y_n}{1-\rho_n \mu_n}, \quad \nabla \mu_n = \frac{\mu_n Y_n - \rho_n\mu_n X_n}{1-\rho_n \mu_n}.
		$$
		Therefore, on $\{(\rho_n,\mu_n)\in B\}$, we obtain
		\begin{multline*}
			\vert \nabla S_n \vert^2 \leq 4 \left(\frac{\rho_n^2 \vert X_n\vert ^2 + \rho_n^2\mu_n^2 \vert Y_n\vert^2 + \mu_n^2 \vert Y_n\vert ^2 + \rho_n^2\mu_n^2 \vert X_n\vert^2}{(1-\rho_n \mu_n)^2} \right)\\
			= 4 \left(\frac{(\rho_n +(\rho_n\mu_n) \mu_n) \rho_n\vert X_n\vert ^2 +  (\mu_n +(\rho_n \mu_n) \rho_n)\mu_n \vert Y_n\vert ^2 }{(1-\rho_n \mu_n)^2} \right).
		\end{multline*}
		Since $\rho_n\mu_n\leq 1$ in the set $B$, we get
		$$
		\vert \nabla S_n \vert^2 \leq 
		4 \left(\frac{S_n \rho_n\vert X_n\vert ^2 + S_n\mu_n \vert Y_n\vert ^2 }{(1-\rho_n \mu_n)^2} \right).
		$$
		Now, using $(1-\rho_n\mu_n) \geq (1 - \frac{S_n^2}{4})$, we finally get, for $(\rho_n,\mu_n)\in B$,
		$$
		\vert \nabla S_n \vert^2 \frac{(1 - \frac{S_n^2}{4})^2}{4 S_n} \leq 
		\rho_n\vert X_n\vert ^2 + \mu_n \vert Y_n\vert ^2 .
		$$
		For $(\rho_n,\mu_n)\in A$, we use $f^{\prime\prime}\geq \frac{r_0}{s}$, where $r_0$ as in Remark \ref{rem:r_0}, and we obtain
		$$
		\int_{(\rho_n,\mu_n)\in A}\frac{1}{S_n} \vert \nabla S_n\vert^2\leq C,
		$$
		for a constant $C$ independent of $n$. Therefore, taking 
		$$
		h(s) := \min \left \{  \frac{\vert 1 - \frac{s^2}{4}\vert}{2 \sqrt{s}} , \frac{1}{\sqrt{s}} \right \},
		$$
		we find that
		$$
		\int_\O h^2(S_n) \vert \nabla S_n\vert^2
		$$
		is bounded independently of $n$. The function $h$ is positive and vanishes only at $s=2$. If we denote by $H$ any anti-derivative of $h$, characterized by $H'=h$, we deduce that $H$ is strictly increasing. Moreover, we obtain a uniform $H^1$ bound on 
		$
		H(S_n)$. 
		This implies that, up to a subsequence, $H(S_n)$ has an a.e. limit. Composing with $H^{-1}$, the same is true for $S_n$. The pointwise limit of $S_n$ can only coincide with its weak limit, i.e. $S$.

		Now, to prove the convergence of $P(\rho_n,\mu_n)$, it is sufficient to write
		$$
		P(\rho_n,\mu_n) = \pi(\rho_n+\mu_n) + \left(P(\rho_n,\mu_n) - \pi(\rho_n
		+\mu_n)  \right)$$
		for any extension of $\pi$ to $\R_+$ ($\pi$ is originally only defined on $[2,+\infty)$, but we can take $\pi=1$ on $[0,2]$).
		
		The function $P(a,b) - \pi(a
		+b)$ is constant, equal to zero, on the set $A$. Then, we can apply Proposition \ref{prop cv B} to obtain the desired convergence. In what concerns 
		$\pi(\rho_n+\mu_n)$ we just need to apply what we just proved on $S_n$.
	\end{proof}
	
	We are now in the position to prove Proposition \ref{th l.s.c. 1}.
	\begin{proof}[Proof of Proposition \ref{th l.s.c. 1}]
		Consider a sequence as in the statement of the proposition. Then, we have
		\begin{align*}
			\int_\O (\vert \nabla f_a(\rho_n,\mu_n)\vert^2 \rho_n + \vert \nabla f_b(\rho_n,\mu_n)\vert^2 \mu_n)(1-\chi(\rho_n,\mu_n)) &\geq \int_\O \frac{\vert \nabla (S_n + P_n)\vert^2 }{S_n}(1-\chi(\rho_n,\mu_n))\\
			&= 4\int_\O \vert \nabla \sqrt{S_n + P_n} \vert^2  \frac{S_n+P_n}{S_n}(1-\chi(\rho_n,\mu_n)).
		\end{align*}
		Since $\sqrt{S_n+P_n}$ is bounded in $H^1$, it converges weakly in $H^1$ up to extraction of a subsequence, and we know that its limit has to be $\sqrt{S+P}$. Using the standard semi-continuity argument we obtain
		\begin{align*}
			\liminf_{n\to+\infty}\int_\O (\vert \nabla f_a(\rho_n,\mu_n)\vert^2 \rho_n + \vert \nabla f_b(\rho_n,\mu_n)\vert^2 \mu_n)(1-\chi(\rho_n,\mu_n)) \geq \int_\O \frac{\vert \nabla (S + P)\vert^2 }{S}(1-\chi(\rho,\mu))\\
			\geq \int_{(\rho,\mu)\in A} (\vert \nabla f_a(\rho,\mu)\vert^2 \rho + \vert \nabla f_b(\rho,\mu)\vert^2 \mu).
		\end{align*}
	\end{proof}
	Therefore, Theorem \ref{th l.s.c.}, that is, the lower semi-continuity of the functional $\Slo$, follows from the previous results as an easy consequence:
	
	\begin{proof}[Proof of Theorem \ref{th l.s.c.}]
		Let $(\rho_n, \mu_n)$ be a sequence satisfying the hypotheses of Theorem \ref{th l.s.c.}. Let $0\leq \chi\leq 1$ be a Lipschitz function compactly supported in $B$. Then
		\begin{align*}
			\Slo(\rho_n,\mu_n) = &\int_\O (\vert \nabla f_a(\rho_n,\mu_n) \vert^2 \rho_n + \vert \nabla f_b(\rho_n,\mu_n)\vert^2 \mu_n )  
			\\= &\int_\O (\vert \nabla f_a(\rho_n,\mu_n) \vert^2 \rho_n + \vert \nabla f_b(\rho_n,\mu_n)\vert^2 \mu_n )\chi(\rho_n,\mu_n) 
			\\ &+ \int_\O (\vert \nabla f_a(\rho_n,\mu_n) \vert^2 \rho_n + \vert \nabla f_b(\rho_n,\mu_n)\vert^2 \mu_n )(1-\chi(\rho_n,\mu_n)).
		\end{align*}
		Applying separately the results of Propositions \ref{th l.s.c. B} and \ref{th l.s.c. 1}, we obtain
		
		$$\liminf_{n\to+\infty} \Slo(\rho_n,\mu_n) \geq \int_\O (\vert \nabla f_a(\rho,\mu) \vert^2 \rho + \vert \nabla f_b(\rho,\mu)\vert^2 \mu )\big(\ind_A(\rho,\mu)+\chi(\rho,\mu) \big).$$
		
		Since $\chi$ is arbitrary, we can select an increasing sequence of cut-off functions $\chi_k$ converging to $\ind_B$, and by monotone convergence we obtain 
		$$\liminf_{n\to+\infty} \Slo(\rho_n,\mu_n) \geq \int_\O (\vert \nabla f_a(\rho,\mu) \vert^2 \rho + \vert \nabla f_b(\rho,\mu)\vert^2 \mu )\big(\ind_A(\rho,\mu)+\ind_B(\rho,\mu) \big)=\Slo(\rho,\mu),$$
		and this concludes the proof.
	\end{proof}
	
	\section{Existence of solutions in the EDI sense}
	\label{sec:EDI}
	
	In this section, we define the JKO scheme for the functional $F$ and three different interpolations (De Giorgi variational, piecewise geodesic, and piecewise constant) for this scheme. We prove that these interpolations converge to the same limit curve and this limit curve is a gradient flow for the functional $F$ in a suitable sense (see Def. \ref{defn:EDI} below). More precisely, we show that the \emph{Energy Dissipation Inequality} \eqref{EDI} holds by using the estimates we obtain via the three interpolations we define subsequently.
	
	First, we define the space $L^2\mathcal H$ as follows: 
	\begin{definition} \label{defn:L^2H}
		The space $L^2\mathcal H$ is composed of all the curves $\rho,\mu:[0,T]\to \mc P (\O)$ such that
		\begin{enumerate} [label=\roman*.]
			\item \label{L^2H: prop1} Both $\rho$ and $\mu$ belong to $AC_2([0,T];W_2(\O))$ (see \cite{AGS08});
			\item \label{L^2H: prop2}For a.e. $t\in [0,T]$ we have $(\rho_t,\mu_t)\in \mc H$;
			\item \label{L^2H: prop3}For every $\eta \in W^{1,\infty}_c(B),$ we have $\int_0^T \|\eta(\rho_t,\mu_t) \|_{H^1(\O)}^2 \d t<+\infty$;
			\item \label{L^2H: prop4} We also have $\int_0^T \| \sqrt{\rho_t+\mu_t} \|_{H^1(\O)}^2 \d t<+\infty$.
		\end{enumerate}
	\end{definition}
	Next we give the following definition:
	\begin{definition}\label{defn:EDI} A pair of  curves $(\rho_t, \mu_t)_{t\in [0,T]} \in L^2\mathcal H$ is \emph{a gradient flow in the EDI sense} for the functional  $F$, given by \eqref{defn:F}, with initial datum $(\rho_0,\mu_0) \in \mc P(\Omega)\times \mc P (\O)$ and $F(\rho_0, \mu_0) < +\infty$, if and only if
		\begin{itemize}
			
			\item There exist a pair of velocities $(v_t, w_t)$, associated with the curves $(\rho_t, \mu_t)$, that satisfy the continuity equations $\p_t \rho_t +\nabla \cdot (\rho_tv_t ) =0$ and  $\p_t\mu_t +\nabla \cdot (\mu_t w_t) =0$ with no-flux boundary conditions.
			\item The pairs $(\rho_t,\mu_t)$ and $(v_t, w_t)$ satisfy the \emph{Energy Dissipation Inequality}
			\begin{align} \label{EDI}
				F(\rho_T, \mu_T) + \frac{1}{2}\int_{0}^{T} \int_\Omega \rho_t |v_t|^2 \d x \d t + \frac{1}{2}\int_{0}^{T} \int_\Omega \mu_t |w_t|^2   \d x \d t + \frac{1}{2}\int_{0}^{T} \Slo(\rho_t, \mu_t) \d t\leq F(\rho_0, \mu_0).
			\end{align}
		\end{itemize}
	\end{definition}
	The main goal of this section is to prove the following theorem:
	\begin{theorem}\label{th:existence}
		For any initial datum $(\rho_0,\mu_0)$ such that  $F(\rho_0, \mu_0) < +\infty$, there exists an EDI gradient flow for the functional $F$.
	\end{theorem}
	
	\subsection{The JKO scheme and the interpolations}
	
	\begin{definition}[JKO scheme for $F$] \label{defn:JKO_scheme}
		For a fixed time step $\tau >0$ (of the form $\tau=T/N$ for some $N\in \N$), we define the JKO scheme as a sequence of probability measures $(\rho_k^{\tau}, \mu_k^{\tau})_k$, with a given initial datum $(\rho_0^\tau,\mu_0^\tau) = (\rho_0,\mu_0)$ such that for $k \in \{0, \cdots, N-1\}$ we have
		\begin{align} \label{defn:JKO}
			(\rho_{k+1}^{\tau}, \mu_{k+1}^{\tau}) \in \argmin_{(\rho, \mu )}  F (\rho, \mu) + \frac{1}{2 \tau} W_2^2 (\rho, \rho_k^{\tau})+ \frac{1}{2 \tau}  W_2^2 (\mu, \mu _k^{\tau}).
		\end{align} 
	\end{definition}
	This sequence of minimizers exists since the functional $F$ is lower semi-continuous for the weak convergence and so is the sum we minimize. Such a sum is also strictly convex, since the measures $\rho_k^{\tau}, \mu_k^{\tau}$ are necessarily absolutely continuous (which implies strict convexity of the Wasserstein terms, see Prop. 7.19 in \cite{S15}): the minimizer is then unique at every step.
	
	Notice that \eqref{defn:JKO} implies that, at each iteration, we have
	\begin{align} \label{ineq:JKO}
		F (\rho_{k+1}^\tau, \mu_{k+1}^\tau ) + \frac{1}{2\tau } W_2^2 ( \rho_{k+1}^\tau, \rho_k^\tau) + \frac{1}{2\tau } W_2^2 ( \mu_{k+1}^\tau, \mu_k^\tau)  \leq 	F (\rho_{k}^\tau, \mu_{k}^\tau ). 
	\end{align}
	
	An important consequence of \eqref{ineq:JKO} is the following inequality, which is obtained by summing over $k$ and using the fact that $F$ is bounded from below:
	\begin{equation}\label{basicH1estimate}
		\frac{1}{2\tau } W_2^2 ( \rho_{k+1}^\tau, \rho_k^\tau) + \frac{1}{2\tau } W_2^2 ( \mu_{k+1}^\tau, \mu_k^\tau)  \leq C:=	F (\rho_0, \mu_0 )-\inf F. 
	\end{equation}

	Our strategy to prove that Inequality \eqref{EDI} holds is to improve \eqref{ineq:JKO} by the help of some interpolations for the sequence $(\rho_k^{\tau}, \mu_k^{\tau})_k$ for the functional $F$. Therefore, we define these interpolations next.
	
	\begin{definition}[De Giorgi variational interpolation] \label{defn:de_Giorgi_interpolation}
		We define the De Giorgi variational interpolation $(\hat \rho_t^\tau, \hat \mu_t^\tau)$ for the sequence $(\rho_k^{\tau}, \mu_k^{\tau})_k$ as follows: for any $s \in (0,1]$ and any $k$, take $ t = (k+s) \tau$ such that 
		\begin{align} \label{interpol_deGiorgi}
			(\hat \rho_{k+s}^{\tau}, \hat \mu_{k+s}^{\tau}) \in \argmin_{(\rho, \mu )}    F (\rho, \mu) + \frac{1}{2 \tau s} W_2^2 (\rho, \rho_k^{\tau})+ \frac{1}{2 \tau s}  W_2^2 (\mu, \mu _k^{\tau}).
		\end{align} 
	\end{definition}
	We notice that when $s=1$, \eqref{interpol_deGiorgi} is nothing but the JKO scheme \eqref{defn:JKO}. The main point (which is now a classical idea in the study of gradient flows, see \cite{AGS08}), is that we can improve \eqref{ineq:JKO} by 
	\begin{multline} \label{eq:de Giorgi1}
		F (\rho_{k+1}^{\tau}, \mu_{k+1}^{\tau})  + \frac{1}{2 \tau} W_2^2 (\rho_{k+1}^{\tau}, \rho_k^{\tau})+ \frac{1}{2 \tau}  W_2^2 (\mu_{k+1}^{\tau}, \mu _k^{\tau})  \\+   \int_{0}^{1}  \frac{W_2^2 ( \hat  \rho_{k+s}^{\tau}, \rho_k^{\tau})}{2 \tau s^2} \d s+  \int_{0}^{1}  \frac{W_2^2 ( \hat  \mu_{k+s}^{\tau}, \mu_k^{\tau})}{2 \tau s^2} \d s \leq   F (\rho_{k}^{\tau}, \mu_{k}^{\tau}).
	\end{multline} 
	To obtain \eqref{eq:de Giorgi1} we define a function $g: [0,1] \to \R $ such that
	\begin{align*}
		g(s) := \min_{(\rho, \mu )}   \,  F (\rho, \mu) + \frac{1}{2 \tau s} W_2^2 (\rho, \rho_k^{\tau})+ \frac{1}{2 \tau s}  W_2^2 (\mu, \mu _k^{\tau}).
	\end{align*} 
	This function is decreasing and hence differentiable a.e. At differentiability points, we necessarily have
	\begin{align*} 
		g^\prime(s)=- \frac{W_2^2 ( \hat  \rho_{k+s}^{\tau}, \rho_k^{\tau})}{2 \tau s^2}  -   \frac{W_2^2 ( \hat  \mu_{k+s}^{\tau}, \mu_k^{\tau})}{2 \tau s^2} .
	\end{align*}
	
	On the other hand, for a monotone function we have an inequality by the fundamental theorem of calculus, which gives here 
	\begin{align*} 
		\int_0^1 g^\prime(s) \d s  \geq g(1)-g(0) = 	F (\rho_{k+1}^{\tau}, \mu_{k+1}^{\tau})  + \frac{1}{2 \tau} W_2^2 (\rho_{k+1}^{\tau}, \rho_k^{\tau})+ \frac{1}{2 \tau}  W_2^2 (\mu_{k+1}^{\tau}, \mu _k^{\tau}) - 	F (\rho_{k}^{\tau}, \mu_{k}^{\tau}).
	\end{align*} Combining the last two lines, we obtain \eqref{eq:de Giorgi1}. Moreover, we have, by the optimality conditions, the following equalities
	\begin{align} \label{deG_optimality}
		\nabla f_a (\hat \rho_{k+s}^{\tau}, \hat \mu_{k+s}^{\tau}) + \frac{\nabla \varphi_{\hat \rho_{k+s}^{\tau}\to\rho_{k}^{\tau}}}{s \tau} = 0 \quad  \hat \rho_{k+s}- \mbox{a.e.} , \quad	\nabla f_b (\hat \mu_{k+s}^{\tau}, \hat \mu_{k+s}^{\tau}) + \frac{\nabla \varphi_{\hat \mu_{k+s}^{\tau}\to\mu_{k}^{\tau}}}{s \tau} = 0 \quad  \hat \mu_{k+s}-\mbox{a.e.},
	\end{align} 
	where $\varphi_{\hat \rho_{k+s}^{\tau}\to\rho_{k}^{\tau}}$ and $\varphi_{\hat \mu_{k+s}^{\tau}\to\mu_{k}^{\tau}}$ are the Kantorovich potentials associated with the transports  from $\hat \rho_{k+s}^\tau$ to $\rho_k^{\tau}$, and from $\hat \mu_{k+s}^\tau$ to $\mu_k^{\tau}$ respectively.  Using \eqref{deG_optimality} and Brenier Theorem (see \cite{B87}), we obtain
	\begin{align} \label{eq:optimality}
		\begin{split}
			\frac{W_2^2 ( \hat  \rho_{k+s}^{\tau}, \rho_k^{\tau})}{(s\tau )^2}  &= \int_\O \hat  \rho_{k+s}^{\tau} \frac{|\nabla \varphi_{\hat \rho_{k+s}^{\tau}\to\rho_{k}^{\tau}}|^2}{(s\tau )^2} \d x  = \int_\O \hat  \rho_{k+s}^{\tau} |\nabla f_a (\hat \rho_{k+s}^\tau,  \hat \mu_{k+s}^\tau)|^2 \d x ,\\
			\frac{W_2^2 ( \hat  \mu_{k+s}^{\tau}, \mu_k^{\tau})}{(s\tau )^2} &= \int_\O \hat  \mu_{k+s}^{\tau} \frac{|\nabla \varphi_{\hat \mu_{k+s}^{\tau}\to\mu_{k}^{\tau}}|^2}{(s\tau )^2} \d x = \int_\O \hat  \mu_{k+s}^{\tau} |\nabla f_b (\hat \rho_{k+s}^\tau,  \hat \mu_{k+s}^\tau)|^2 \d x .
		\end{split} 
	\end{align} 
	Using \eqref{eq:optimality}, Inequality \eqref{eq:de Giorgi1} re-writes
	\begin{multline}  \label{ineq:De Giorgi}
		F (\rho_{k+1}^{\tau}, \mu_{k+1}^{\tau})  + \frac{1}{2 \tau} W_2^2 (\rho_{k+1}^{\tau}, \rho_k^{\tau})+ \frac{1}{2 \tau}  W_2^2 (\mu_{k+1}^{\tau}, \mu _k^{\tau})  \\+   \frac{\tau}{2}\int_{0}^{1}\int_\O \hat \rho_{k+s}  |\nabla f_a (\hat \rho_{k+s}^\tau, \hat \mu_{k+s}^\tau)\vert^2 \d x \d s +  \frac{\tau}{2}\int_{0}^{1}\int_\O \hat \mu_{k+s}  | \nabla f_b (\hat \rho_{k+s}^\tau, \hat \mu_{k+s}^\tau)\vert^2 \d x \d s \leq 	F (\rho_{k}^{\tau}, \mu_{k}^{\tau}).
	\end{multline} 
	
	\begin{definition}[Piecewise constant interpolation]
		\label{defn:Piecewise_constant_interpolation}
		We define the piecewise constant interpolation as a pair of piecewise constant curves $(\bar \rho_t^{\tau},  \bar \mu_t^{\tau}) $ and a pair of velocities $(\bar v_t^{\tau}, \bar w_t^{\tau} ) $ associated with these piecewise constant curves such that for every $t \in (k\tau, (k+1) \tau]$ and $k \in \{ 0, \cdots, N-1\}$, $N \in \N$, they satisfy
		\begin{align} \label{PC_curves}
			\begin{split}
				(\bar \rho_t^{\tau},  \bar \mu_t^{\tau}) &= (\rho_{k+1}^{\tau}, \mu_{k+1}^{\tau}), \\
				(\bar v_t^{\tau}, \bar w_t^{\tau} ) &= (v_{k+1}^{\tau}, w_{k+1}^{\tau}) =  \bigg (\frac{\mbox{id}- T_{\rho_{k+1}\to \rho_k }^\tau}{\tau}, \frac{\mbox{id}- T_{\mu_{k+1} \to \mu_k }^\tau }{\tau}\bigg)
			\end{split}
		\end{align} where $T_{\rho_{k+1}\to \rho_k^\tau }^\tau $ and $T_{\mu_{k+1} \to \mu_k }^\tau $ are the optimal transport maps from $\rho_{k+1}^\tau$ to $\rho_k^\tau$ and from $\mu_{k+1}^\tau$ to $\mu_k^\tau$ respectively. We define the momentum variables by $\bar E^\tau (t) :=(\bar E_\rho^\tau (t), \bar E_\mu^\tau (t)) = (\bar \rho_t^\tau  \bar v_t^\tau , \bar \mu^\tau_t \bar w_t^\tau  )$.
	\end{definition}
	
	\begin{definition}[Piecewise geodesic interpolation]
		\label{defn:Piecewise_geo_interpolation}
		We define the piecewise geodesic interpolation as a pair of densities $(\tilde{\rho}^\tau_t, \tilde{ \mu}^\tau_t)$ that interpolate the discrete values $(\rho_k^\tau, \mu_k^\tau)_k$ along Wasserstein geodesics: for each $k$ we define
		\begin{align} \label{geo_curves}
			(\tilde{\rho}_t^\tau, \tilde{ \mu}_t^\tau  ) =  \left ( \left (  \mbox{id}-(k\tau - t) v_k^\tau\right)_\# \rho_k^\tau, \left (  \mbox{id}-(k\tau - t) w_k^\tau \right)_\# \mu_k^\tau\right ), \quad \mbox{for } t \in ((k-1)\tau, k \tau].
		\end{align} 
		We also define some velocity fields $(\tilde{v}_t^\tau, \tilde{w}_t^\tau )$ as follows
		$$\tilde{v}_t^\tau=v_k^\tau\circ  \left (  \mbox{id}-\frac{k\tau - t}{\tau} v_k^\tau\right)^{-1},\;
		\tilde{w}_t^\tau=w_k^\tau\circ  \left (  \mbox{id}-\frac{k\tau - t}{\tau} v_k^\tau\right)^{-1},
		\quad \mbox{for } t \in ((k-1)\tau, k \tau).
		$$
		In this way we can check that we have 
		$$
		\partial_t \tilde{\rho_t}^\tau+\nabla\cdot(\tilde{\rho_t}^\tau \tilde v_t^\tau)=0, \quad \partial_t \tilde{\mu}_t^\tau+\nabla\cdot(\tilde{\mu}_t^\tau \tilde w_t^\tau)=0,
		$$
		and, for $t\in ((k-1)\tau,k\tau)$, we also have
		\begin{align*}
			\|\tilde v_t^\tau \|_{L^2 (\tilde \rho_t^\tau)} = \vert (\tilde \rho_t^\tau)'  \vert (t)=\frac{W_2(\rho_{k-1}^\tau,\rho_k^\tau)}{\tau}, \quad \mbox{and} \quad 	\|\tilde w_t^\tau \|_{L^2 (\tilde \mu_t^\tau)} = \vert (\tilde \mu_t^\tau)'  \vert (t)=\frac{W_2(\mu_{k-1}^\tau,\mu_k^\tau)}{\tau}.
		\end{align*} 
		
		We also define the momentum variables by $\tilde E^\tau (t) = (\tilde E_\rho^\tau(t), \tilde E_\mu^\tau(t)) := (\tilde \rho_t^\tau \tilde v_t^\tau , \tilde \mu_t^\tau \tilde  w_t^\tau )$. 
	\end{definition}

	\medskip
	
	Now that we have defined the three interpolations above and that we have improved \eqref{ineq:JKO} into \eqref{ineq:De Giorgi}, we would like to replace the terms involving the De Giorgi variational interpolation $(\hat \rho^\tau_{k+s}, \hat \mu^\tau_{k+s})$ in \eqref{ineq:De Giorgi} by the slope functional $\Slo (\hat \rho^\tau_{k+s}, \hat \mu^\tau_{k+s})$. To be able to do that, first we prove two technical lemmas (Lemmas \ref{lem S^2 over S bound} and \ref{lem:energy_estimate} below) and then by using them we prove that the pair of curves obtained via the De Giorgi variational interpolation belongs to the space $\mc H$.

	\begin{lemma} \label{lem S^2 over S bound}
		For every $\rho, \mu \in C^\infty(\O)$ strictly positive densities we have
		\begin{align} \label{bound S^2 over S}
			\int_\O \left ( \nabla \rho \nabla f_a (\rho, \mu) +  \nabla \mu \nabla f_b (\rho, \mu ) \right ) \geq r_0\int_\O \frac{|\nabla S|^2}{S},
		\end{align} where $S:= \rho + \mu$ and $r_0>0$ is defined in Remark \ref{rem:r_0} in Section \ref{sec:convexification}.
	\end{lemma}
	\begin{proof}
		We consider the sets $A$ and $B$ separately. For $(\rho, \mu) \in A$, we have
		$$
		\nabla \rho \nabla f_a (\rho, \mu) + \nabla \mu \nabla f_b (\rho, \mu ) = \lvert \nabla S\rvert^2 \tilde f''(S) = \lvert \nabla S\rvert^2 \frac{1- \pi(S)}{S - 2 \pi(S)}\geq r_0\frac{|\nabla S|^2}{S}.$$
		For $(\rho, \mu) \in B$, we have 
		$$
		\nabla \rho \nabla f_a (\rho, \mu)  + \nabla \mu \nabla f_b (\rho, \mu )  = \frac{ \vert \nabla \rho \vert ^2}{\rho} + 2 \nabla \rho \nabla \mu + \frac{\vert \nabla \mu \vert ^2}{\mu }.
		$$
		We want to show that the inequality
		\begin{align*}
			\frac{ \vert \nabla \rho \vert ^2}{\rho} + 2 \nabla \rho \nabla \mu + \frac{\vert \nabla \mu \vert ^2}{\mu } \geq r_0 \frac{ \vert \nabla \rho + \nabla \mu \vert^2}{S}
		\end{align*} is always satisfied. This means we want to have
		\begin{align*}
			\vert \nabla \rho \vert^2 \Big (\frac{1}{\rho} - \frac{r_0}{S}\Big) + 	\vert \nabla \mu\vert^2 \Big(\frac{1}{\mu} - \frac{r_0}{S}\Big)  + 2 \nabla \rho \nabla \mu \left(1 - \frac{r_0}{S}\right) \geq 0.
		\end{align*}
		We now look at the matrix
		
		$$
		\left(	\begin{array}{cc}
			\frac{1}{\rho} - \frac{r_0}{S} & 1 - \frac{r_0}{S} \\ 1 - \frac{r_0}{S} & 	\frac{1}{\mu} - \frac{r_0}{S}
		\end{array}\right),
		$$
		and prove that it is positive definite. Using $r_0\leq 1$ and $\rho,\mu\leq S$ we easily see that the terms on the diagonal are non-negative. We then compute the determinant and obtain 
		
		$$\frac{1}{\rho \mu} - \frac{r_0}{S} \left( \frac{1}{\rho} + \frac{1}{\mu} - 2\right) - 1  = \frac{1-r_0}{\rho \mu} - 1 +  \frac{2r_0}{S}.$$
		For fixed sum $S=\rho+\mu$ this last expression is minimal if the product $\rho\mu$ is maximal. If $S>2$ we can then bound this from below by $\frac{1-r_0}{\pi(S)}-1+\frac{2r_0}{S}$. This quantity is non-negative as a consequence of the definition of $r_0$ since we have
		$$r_0\leq s\frac{1-\pi}{s-2\pi}\Leftrightarrow \frac{1-r_0}{\pi}-1+\frac{2r_0}{s}\geq 0.$$
		If $S\leq 2$ we just use $\rho\mu\leq S^2/4\leq S/2$ and bound the same quantity from below by $2\frac{1-r_0}{S}-1+\frac{2r_0}{S}=\frac 2S-1\geq 0.$
	\end{proof}

	\begin{lemma} \label{lem:energy_estimate}
		Let us consider the auxiliary energy functional $G(\rho, \mu) = \int_\O g(\rho, \mu)$ where $g(a,b) := a\log a + b\log b$ (defined as equal to $+\infty$ if measures are not absolutely continuous). Suppose that $\rho_0, \mu_0$ are given absolutely continuous measures and let us call  $(\rho_1, \mu_1)$ the unique solution of
		$$(\rho_1,\mu_1)=\argmin_{(\rho,\mu)} F(\rho,\mu)+\frac{W_2^2(\rho,\rho_0)}{2\tau}+\frac{W_2^2(\mu,\mu_0)}{2\tau}.$$
		Then, setting $S_1:=\rho_1+\mu_1$, we have 
		\begin{align*}
			G (\rho_0, \mu_0) - G(\rho_1, \mu_1) \geq r_0 \tau \int_\O \frac{|\nabla S_1|^2}{S_1},
		\end{align*}
		where the number $r_0>0$ is defined in Remark \ref{rem:r_0} in Section \ref{sec:convexification}.
	\end{lemma}
	\begin{proof}
		We prove this result via a flow-interchange inequality and a regularization argument.
		We fix $\ve >0$ and define $F_\ve (\rho, \mu) := F (\rho, \mu)+ \ve G (\rho, \mu) $. For a given sequence of smooth measures $\rho_{0,\ve}, \mu_{0,\ve} \in C^\infty $ we consider the sequence of minimizers
		$$(\rho_{1,\ve},\mu_{1,\ve})=\argmin_{(\rho,\mu)} F_\ve(\rho,\mu)+\frac{W_2^2(\rho,\rho_{0,\ve})}{2\tau}+\frac{W_2^2(\mu,\mu_{0,\ve})}{2\tau}.$$
		
		If we write the optimality conditions for the above minimization problem we have $\rho_{1,\ve},\mu_{1,\ve}>0$ (the argument is the same as in the proof of Lemma 8.6 in \cite{S15}) and 
		$$f_a(\rho_{1,\ve},\mu_{1,\ve})+\ve \log\rho_{1,\ve}+\frac{\varphi_{\rho_{1,\ve}\to\rho_{0,\ve}}}{\tau}=C,\quad f_b(\rho_{1,\ve},\mu_{1,\ve})+\ve \log\mu_{1,\ve}+\frac{\varphi_{\mu_{1,\ve}\to\mu_{0,\ve}}}{\tau}= \tilde C,$$
		where $C, \tilde C $ are some positive constants. Since $(a,b)\mapsto (f_a(a,b)+\ve\log a,f_b(a,b)+\ve\log b)$ is the gradient of a strictly convex function it is a diffeomorphism and we deduce that $\rho_{1,\ve}$ and $\mu_{1,\ve}$ have the same regularity of the Kantorovich potential, and are Lipschitz continuous. They are also bounded from below because of the logarithm in the optimality conditions and Caffarelli's theory (see section 4.2.2 in \cite{Villani}) implies that the Kantorovich potentials are $C^{2,\alpha}$ (we assumed that we are in a convex domain). Iterating these regularity arguments gives that  $\rho_{1,\ve}$ and $\mu_{1,\ve}$ are $C^\infty$ functions.
		
		We then use the geodesic convexity of the entropy to deduce that we have
		\begin{align*}
			G(\rho_{0,\ve}, \mu_{0,\ve}) - G(\rho_{1,\ve}, \mu_{1,\ve}) \geq \frac{\d}{\d s} G (\rho_s, \mu_s ) \bigg | _{s=0} =\left( \int_\O \p_s \rho_s \log \rho_s  + \int_\O  \p_s \mu_s \log \mu_s \right) \bigg | _{s=0}, 
		\end{align*}
		where $(\rho_s, \mu_s)$ is a pair of geodesic curves in $W_2(\Omega) \times W_2 (\O)$ connecting the densities $(\rho_{1,\ve}, \mu_{1,\ve})$ to the densities $(\rho_{0,\ve}, \mu_{0,\ve})$ (pay attention that, in a JKO scheme, this interpolation starts from the new points and go back to the old points). We then use the continuity equations $\p_s \rho_s + \nabla \cdot (\rho_s v_s)=0$, and $\p_s \mu_s + \nabla \cdot (\mu_s w_s)=0$, together with the fact that the initial velocity fields $v_0$ and $w_0$ can be obtained as the opposite of the gradient of the corresponding Kantorovich potentials. Hence we have
		\begin{multline} \label{G_entropy}
			G(\rho_{0,\ve}, \mu_{0,\ve}) - G(\rho_{1,\ve}, \mu_{1,\ve}) 
			\geq - \int_\O  \nabla \cdot (\rho_0 v_0) \log \rho_0 - \int_\O \nabla \cdot (\mu_0 w_0) \log \mu_0 
			\\ =\int_\O \nabla \rho_0 \cdot v_0 +\int_\O \nabla \mu_0 \cdot w_0
			= - \int_\O \nabla \rho_{1,\ve} \cdot  \nabla \varphi_{\rho_{1,\ve}\to\rho_{0,\ve}} - \int_\O \nabla \mu_{1,\ve}  \cdot \nabla\varphi_{\mu_{1,\ve}\to\mu_{0,\ve}}.
		\end{multline}
		Using the optimality conditions we obtain
		\begin{align*}	
			G(\rho_{0,\ve}, \mu_{0,\ve}) - G(\rho_{1,\ve}, \mu_{1,\ve})\geq & \tau \int_\Omega\nabla\rho_{1,\ve}\cdot\nabla f_a(\rho_{1,\ve}, \mu_{1,\ve})+
			\ve\tau\int_\Omega\nabla\rho_{1,\ve}\cdot\nabla\log(\rho_{1,\ve})\\
			+&\tau\int_\Omega\nabla\mu_{1,\ve}\cdot\nabla f_b(\rho_{1,\ve}, \mu_{1,\ve})+
			\ve\tau\int_\Omega \nabla\mu_{1,\ve}\cdot\nabla\log(\mu_{1,\ve}).
		\end{align*}
		Dropping the positive terms with the gradients of the logarithms and applying Lemma \ref{lem S^2 over S bound} we then obtain
		$$	G(\rho_{0,\ve}, \mu_{0,\ve}) - G(\rho_{1,\ve}, \mu_{1,\ve})\geq r_0 \tau \int_\O \frac{|\nabla S_{1,\ve}|^2}{S_{1,\ve}},$$
		where $S_{1,\ve}=\rho_{1,\ve}+\mu_{1,\ve}$. It is then enough to let $\ve\to 0$ if we choose an approximation $\rho_{0,\ve},\mu_{0,\ve}$ s.t. $G(\rho_{0,\ve},\mu_{0,\ve})\to G(\rho_0,\mu_0)$. Note that the terms in $G(\rho_{1,\ve},\mu_{1,\ve})$ and $S_{1,\ve}$ are lower semi-continuous for the weak convergence (and the sequence $(\rho_{0,\ve},\mu_{0,\ve})$ weakly converges  to $(\rho_1,\mu_1)$ by $\Gamma-$convergence of the minimized functionals and uniqueness of the minimizer at the limit).
	\end{proof}

	\begin{lemma} \label{lem:de_Giorgi_H^1}
		Let $(\hat \rho_t^\tau, \hat \mu_t^\tau)$ be the pair of curves obtained by the De Giorgi variational interpolation. Then $ (\hat \rho_t^\tau, \hat \mu_t^\tau) \in \mc H$ for every $t$.
	\end{lemma}
	\begin{proof}
		First, we note that the optimality conditions provide Lipschitz continuity, for fixed $\tau$, for $f_a (\hat \rho_t^\tau, \hat \mu_t^\tau)$ and $f_a (\hat \rho_t^\tau, \hat \mu_t^\tau)$, and hence for $\chi(\hat \rho_t^\tau, \hat \mu_t^\tau)$  (see Remark \ref{rem:local_Lipschitz} in Section \ref{sec:convexification}). Therefore Property \ref{H:prop1} of Def. \ref{defn: H} is satisfied.
		
		For Property  \ref{H:prop2} of Def. \ref{defn: H}, we apply Lemma \ref{lem:energy_estimate} with $s\tau$ instead of $\tau$, which guarantees the $H^1$ behavior of the square root of the sum.
	\end{proof}
	\begin{remark}
		Note that the very same argument implies that the curves obtained by the piecewise constant interpolation also belong to the space $\mc H$.
	\end{remark}

	\subsection{Existence of solutions}
	The goal of this section is to prove Theorem \ref{th:existence}. 	
	We start by proving the following lemma:
	\begin{lemma}
		\label{lem:limiti curve}
		The pair of curves $(\hat \rho_t^{\tau}, \hat \mu_t^{\tau})$, $(\bar \rho_t^\tau , \bar \mu_t^\tau )$ and  $(\tilde \rho_t^\tau, \tilde \mu_t^\tau)$ given by Definitions \ref{defn:de_Giorgi_interpolation}, \ref{defn:Piecewise_constant_interpolation} and \ref{defn:Piecewise_geo_interpolation} respectively  converge up to subsequences, as $\tau \to 0$, to the same limit curve $(\rho_t, \mu_t)$ uniformly in $W_2$ distance. Moreover, the vector-valued measure $\tilde E^\tau$ corresponding to the momentum variable of the piecewise geodesic interpolations, also converges weakly-* in the sense of measures on $[0,T]\times\Omega$ to a limit vector measure $E$ along the same subsequence.
	\end{lemma} 
	
	\begin{proof}
		Recall that the geodesic speed is constant on each interval $(k\tau, (k+1)\tau)$, and this implies
		\begin{align*} 
			\|\tilde v_t^\tau \|_{L^2(\tilde \rho_t^\tau)}   = \vert (\tilde \rho_t^\tau)'  \vert (t) = \frac{W_2(\rho_{k+1}, \rho_k)}{\tau } = \frac{1}{\tau } \int_\O \vert \mbox{id} - T_{k+1}^\tau \vert^2 \rho_{k+1}^\tau  = \|v_{k+1}^\tau\|_{L^2(\rho_{k+1}^\tau)},
		\end{align*} and similar for $\tilde w_t^\tau$. Then we obtain
		\begin{align*}
			\|\tilde v_t^\tau \|_{L^2(\tilde \rho_t^\tau)} = \|v_{k+1}^\tau\|_{L^2(\rho_{k+1}^\tau)} = \|\bar v_t^\tau \|_{L^2(\bar \rho_t^\tau)} \quad \mbox{and} \quad \|\tilde w_t^\tau \|_{L^2(\tilde \mu_t^\tau)} = \|w_{k+1}^\tau\|_{L^2(\mu_{k+1}^\tau)} = \|\bar w_t^\tau \|_{L^2(\bar \mu_t^\tau)}.
		\end{align*} 
		
		Let us note that we have 
		\begin{equation}\label{L2vel}
			\int_0^T |(\tilde\rho_t^\tau)'|(t)^2 \d t= \int_0^T \|\tilde v_t^\tau\|^2_{L^2(\tilde \rho_t^\tau )} \d t =\sum_k \tau \left( \frac{W_2(\rho_{k+1}, \rho_k)}{\tau } \right)^2 \leq C,
		\end{equation}
		where the inequality is a consequence of \eqref{basicH1estimate}.
		
		We first use this inequality to estimate the momentum variables, since we have 
		$$
		\int_0^T \int_\O \vert \tilde E^\tau_\rho (t) \vert  \d t =  	\int_0^T \|\tilde v_t^\tau \|_{L^1(\tilde \rho_t^\tau )} \d t 
		\leq \sqrt{T} \int_0^T \|\tilde v_t^\tau\|^2_{L^2(\tilde \rho_t^\tau )} \d t \leq C.
		$$
		
		Analogous estimates can be obtained for $\tilde E ^\tau_\mu$. This means that $\tilde E^\tau$ is bounded in $L^1 ([0,T] \times \Omega)$ and we obtain the weak-* compactness in the space of measures on space-time.
		
		\medskip 
		
		It is now classical in gradient flows, as a consequence of the estimate on the $L^2$ norm of the velocities \eqref{L2vel}, to obtain Hölder bounds on the geodesic interpolations. Indeed, the pair $(\tilde{\rho}^\tau_t, \tilde{ \mu}^\tau_t)$ is uniformly $\frac12-$Hölder continuous since by using the previous computations we can show that, for $s<t$,
		\begin{align*}
			W_2 (\tilde{\rho}^\tau_t, \tilde{\rho}^\tau_s) \leq  \sqrt{t-s}\left(\int_{t_1}^{t_2} |(\tilde\rho^\tau)'|(t)^2 \d t \right)^{1/2}\leq C\sqrt{t-s}, 
		\end{align*} 
		and an analogous estimate holds for $\tilde\mu^\tau$. Since the domain of the curves $\tilde \rho, \tilde \mu : [0,T] \to W_2(\Omega)$ is compact and so is the image domain $W_2(\Omega)$, we can pass to the limit by using Ascoli-Arzelà theorem. Therefore there exists a subsequence $\tau_j \to 0$ such that 
		\begin{align} \label{limit_E}
			\begin{split}
				\tilde E_\rho^{\tau_j} \to E_\rho \quad &\mbox{and}\quad \tilde E_\mu^{\tau_j} \to \tilde  E_\mu \quad \mbox{weakly-* as measures},\\ 
				\tilde \rho_t^{\tau_j} \to \rho_t  \quad &\mbox{and}\quad  \tilde \mu_t^{\tau_j} \to  \mu_t  \quad \mbox{ uniformly in } W_2. 
			\end{split}
		\end{align}
		Moreover the curves $(\bar \rho_t^\tau, \bar \mu_t^\tau )$ obtained from the piecewise constant interpolation converge uniformly to the same limit curve $(\rho_t, \mu_t)$ as the ones obtained from the piecewise geodesic interpolation since we have that
		\begin{align*}
			W_2 (\bar {\rho}_t^\tau, \tilde{\rho}^\tau_t) \leq C \sqrt{\tau} \quad \mbox{and} \quad 	W_2 (\bar {\mu}_t^\tau , \tilde{ \mu}^\tau_t) \leq \tilde C \sqrt{\tau},
		\end{align*} where $C, \tilde C$ are some positive constant. This is shown by using that $(\tilde \rho_t^\tau, \tilde \mu_t^\tau)$ and $(\bar \rho_t^\tau, \bar \mu_t^\tau)$ coincide at every $t =k \tau$ and $(\bar \rho_t^\tau, \bar \mu_t^\tau)$ are constant in each interval $(k\tau, (k+1)\tau]$. The details of these computations can be found in Section $8.3$ of \cite{S15}. 
		
		\medskip 
		
		For the pair of curves $(\hat  \rho_t^\tau, \hat \mu_t^\tau)$ defined by De Giorgi variational interpolation we have
		\begin{multline*}
			W_2 (\hat \rho_t^\tau , \tilde \rho_t^\tau) \leq W_2 (\hat \rho_{k+s}^\tau, \bar \rho_t^\tau) + W_2 ( \bar \rho_t^\tau , \tilde \rho_t^\tau)  =  W_2 (\hat \rho_{k+s}^\tau, \rho_{k+1}^\tau) + W_2 ( \bar \rho_t^\tau , \tilde \rho_t^\tau) 
			\\  \leq  W_2 (\hat \rho_{k+s}^\tau, \rho_{k}^\tau) + W_2 (\rho_{k}^\tau, \rho_{k+1}^\tau)  + W_2 ( \bar \rho_t^\tau , \tilde \rho_t^\tau)  \leq C \sqrt{\tau},
		\end{multline*} and similarly $	W_2 (\hat \mu_t^\tau , \tilde \mu_t^\tau) \leq \tilde C \sqrt{\tau}$ for some positive constants $C ,\tilde C$. This implies as $\tau \to 0$,
		\begin{align*}
			\hat \rho^\tau_t \to \rho_t  \quad &\mbox{and}\quad  \hat \mu^{\tau}_t\to  \mu_t  \quad \mbox{ uniformly in } W_2. 
		\end{align*} 	
		
		This means that the pair $(\hat  \rho_t^\tau, \hat \mu_t^\tau)$ also converges to the limit curves $(\rho_t, \mu_t)$ uniformly in $[0,T]$. Therefore we showed that the three interpolations that are defined for \eqref{defn:JKO} converge to the same limit curves $(\rho_t, \mu_t)$. 
	\end{proof}
	By Lemma \ref{lem:de_Giorgi_H^1}, we know that, for the curves obtained by the De Giorgi intepolation we have $(\hat \rho_{t}^\tau, \hat \mu_{t}^\tau) \in \mc H$, for all $t$. Then we have the following: 
	\begin{align*}
		\int_\O \hat \rho_{k+s}  |\nabla f_a (\hat \rho_{k+s}^\tau, \hat \mu_{k+s}^\tau)\vert^2 +  \int_\O \hat \mu_{k+s}  | \nabla f_b (\hat \rho_{k+s}^\tau, \hat \mu_{k+s}^\tau)\vert^2 =  \Slo (\hat \rho_{k+s}^\tau, \hat \mu_{k+s}^\tau).
	\end{align*}
	Using the above equality, \eqref{ineq:De Giorgi} re-writes as 
	\begin{multline} \label{discrete EDI} 
		F (\rho_{k+1}^{\tau}, \mu_{k+1}^{\tau})  + \frac{1}{2 \tau} W_2^2 (\rho_{k+1}^{\tau}, \rho_k^{\tau})+ \frac{1}{2 \tau}  W_2^2 (\mu_{k+1}^{\tau}, \mu _k^{\tau})  +    \frac{\tau}{2}\int_{0}^{1}\Slo(\hat \rho_{k+s}^\tau, \hat \mu_{k+s}^\tau) \d s  \leq 	F (\rho_{k}^{\tau}, \mu_{k}^{\tau}).
	\end{multline} 
	On the other hand since we have
	\begin{align*}
		\frac{W^2_2(\rho_{k+1},\rho_k)}{\tau} = \int_{k\tau}^{(k+1)\tau} \int_\Omega \tilde \rho_t^\tau \vert \tilde v_t^\tau \vert^2 \d x \d t \quad \mbox{and} \quad \frac{W^2_2(\mu_{k+1},\mu_k)}{\tau} = \int_{k\tau}^{(k+1)\tau} \int_\Omega \tilde \mu_t^\tau \vert \tilde w_t^\tau \vert^2 \d x \d t, 
	\end{align*} we can re-write \eqref{discrete EDI} as
	\begin{multline*}
		F (\rho_{k+1}^{\tau}, \mu_{k+1}^{\tau})  + \frac{1}{2 } \int_{k\tau}^{(k+1)\tau} \int_\Omega \tilde \rho_t^\tau \vert \tilde v_t^\tau \vert^2 \d x \d t  + \frac{1}{2 } \int_{k\tau}^{(k+1)\tau} \int_\Omega \tilde \mu_t^\tau \vert \tilde w_t^\tau \vert^2 \d x \d t \\  + \frac{\tau}{2}\int_{0}^{1}\Slo(\hat \rho_{k+s}^\tau, \hat \mu_{k+s}^\tau) \d s  \leq 	F (\rho_{k}^{\tau}, \mu_{k}^{\tau}).
	\end{multline*}
	and, up to a change of variable, we obtain
	\begin{multline} \label{discrete EDI 2}
		F (\rho_{k+1}^{\tau}, \mu_{k+1}^{\tau})  +\frac{1}{2 } \int_{k\tau}^{(k+1)\tau} \int_\Omega\tilde \rho_t^\tau \vert \tilde v_t^\tau \vert^2 \d x \d t  + \frac{1}{2 } \int_{k\tau}^{(k+1)\tau} \int_\Omega  \tilde \mu_t^\tau \vert \tilde w_t^\tau \vert^2 \d x \d t \\+   \frac{1}{2}\int_{k\tau}^{(k+1)\tau} \Slo(\hat \rho_{t}^\tau,\hat \mu_{t}^\tau) \d t  \leq 	F (\rho_{k}^{\tau}, \mu_{k}^{\tau}).
	\end{multline}
	Next, we prove the following two results which will be helpful in the proof of Theorem \ref{th:existence}:
	
	\begin{prop} \label{prop:F_l.s.c._t}
		Suppose that a sequence of curves $(\rho_t^n, \mu_t^n) $ satisfies $(\rho_t^n,\mu_t^n)\in \mathcal{H}$ for every $n$ and a.e. $t$, and
		\begin{itemize}
			\item $\int_0^T \Slo (\rho_t^n, \mu_t^n) \d t \leq  C$ for all $n$.
			\item $(\rho_t^n, \mu_t^n) \rightharpoonup (\rho_t, \mu_t)$ for each $t$, where $\rho_t$ and $\mu_t$ are curves in $AC_2([0,T];W_2(\Omega))$.
			\item $\int_0^T \|\sqrt{\rho_t + \mu_t}\|_{H^1(\O)}^2 \d t < +\infty$.
		\end{itemize} Then we have $(\rho_t, \mu_t) \in L^2\mathcal{H}$ and $ \Slo  (\rho_t, \mu_t)\leq \underset{n \to \infty}{\liminf  }\Slo  (\rho_t^n, \mu_t^n)$.
	\end{prop}
	\begin{proof}
		First, we consider the function $t\mapsto \liminf_n \Slo  (\rho_t^n, \mu_t^n)$. Fatou's lemma implies that
		$$\int_0^T  \liminf_{n \to +\infty} \Slo  (\rho_t^n, \mu_t^n)\d t\leq \liminf_{n \to +\infty} \int_0^T  \Slo  (\rho_t^n, \mu_t^n)\d t\leq C<+\infty.$$
		In particular, the liminf of the slope is finite for a.e. $t$. If we take a function $\chi\in W_c^{1,\infty} (B)$, then we have
		\begin{equation}\label{Hslo}
			\|\chi  (\rho_t^n, \mu_t^n) \|_{H^1(\O)}^2 \lesssim   \Slo (\rho_t^n, \mu_t^n).
		\end{equation}
		Using Proposition \ref{prop cv B} we have
		\begin{align*}
			\chi  (\rho_t^n, \mu_t^n) \underset{n \to + \infty}{\longrightarrow}\chi (\rho_t, \mu_t) \quad \mbox{strongly in $L^2$ and weakly in $H^1$}.
		\end{align*} 
		Hence, we can pass to the limit as $n\to + \infty$ and deduce for a.e. $t$ that we have $\chi  (\rho_t, \mu_t) \in H^1(\O)$, i.e. Property \ref{H:prop1} for being in $\mc H$. Moreover, the assumption on $\sqrt{\rho_t+\mu_t}$ provides Property \ref{H:prop2} for a.e. $t$. Hence, we know $ (\rho_t, \mu_t)\in \mc H$ for a.e. $t$. We now use the lower semi-continuity of the slope on $\mc H$ together with Fatou's lemma to deduce 
		\begin{align*}
			\int_0^T \Slo (\rho_t, \mu_t) \d t \leq \int_0^T \liminf_{n \to + \infty} \, \Slo (\rho_t^n, \mu_t^n) \d t\leq  \liminf_{n \to + \infty}  \int_0^T \Slo  (\rho_t^n, \mu_t^n) \d t\leq C.
		\end{align*}
		Using again \eqref{Hslo} we also obtain 
		\begin{align*}
			\int_0^T \|\chi  (\rho_t^n, \mu_t^n) \|_{H^1(\O)}^2 \d t \lesssim   \int_0^T \Slo (\rho_t^n, \mu_t^n) \d t \leq C, 
		\end{align*} 
		which, combined with the $L^2$ integrability of the $H^1$ norm of the square root, provides $(\rho,\mu)\in L^2 \mc H$. 
	\end{proof}
	\begin{lemma} \label{lem:sqrt_S_L2H1}
		Let $S_t:=\rho_t+\mu_t$ where $(\rho,\mu)$ is the limit of the piecewise constant interpolation of the JKO scheme. We then have $ \int_{0}^{T} \|\sqrt{S_t}\|_{H^1(\Omega)}^2 \d t <+\infty$.
	\end{lemma}
	\begin{proof} We iterate the estimate of Lemma \ref{lem:energy_estimate}, which guarantees 
		\begin{align*}
			G (\rho_k^\tau, \mu_k^\tau) - G(\rho_{k+1}^\tau, \mu_{k+1}^\tau) \geq r_0 \tau \int_\O \frac{|\nabla S_{k+1}^\tau|^2}{S_{k+1}^\tau},
		\end{align*}
		where $S_k^\tau:=\rho_k^\tau+\mu_k^\tau$, and where $G$ is defined in Lemma \ref{lem:energy_estimate}. Summing over $k$, we obtain, for $N=T/\tau$,
		$$
		G (\rho_0, \mu_0) \geq  G(\rho_{N}^\tau, \mu_N^\tau)+ r_0 \int_0^T \int_\O \frac{|\nabla \bar S^\tau|^2}{\bar S^\tau}, 
		$$ where $\bar S^\tau := \bar \rho^\tau + \bar \mu^\tau$, and $\bar \rho^\tau, \bar \mu^\tau$ are the curves obtained from the piecewise constant interpolation. It is then enough to pass to the limit $\tau\to 0$ and apply the semi-continuity of the terms on the right hand side above to obtain 
		$$r_0\int_0^T\int_\Omega \frac{|\nabla \bar S|^2}{\bar S}\leq G(\rho_0,\mu_0)-G(\rho_T,\mu_T),$$
		which proves the claim since $|\nabla \sqrt{S}|^2=\frac{|\nabla S|^2}{4S}$.
	\end{proof}
	
	Now we have the necessary tools to prove Theorem \ref{th:existence}. 
	\begin{proof}[Proof of Theorem \ref{th:existence}]
		Let $(\hat \rho_t^{\tau}, \hat \mu_t^{\tau})$, $(\bar \rho_t^\tau, \bar \mu_t^\tau )$ and  $(\tilde \rho_t^\tau, \tilde \mu_t^\tau)$ be the curves obtained by the De Giorgi variational interpolation, the piecewise constant interpolation and the geodesic interpolation respectively. 
		Summing \eqref{discrete EDI 2} over $k$ we obtain 
		\begin{align} \label{discrete EDI 3}
			F (\rho_{T}^{\tau}, \mu_{T}^{\tau}) + \frac{1}{2 } \int_{0}^{T}\int_\O\tilde \rho_t^\tau \vert   \tilde{v_t}^\tau\vert^2 \d x \d t + \frac{1}{2 }  \int_{0}^{T}\int_\O  \tilde \mu_t^\tau \vert \tilde{w_t}^\tau\vert^2  \d x \d t + \frac{1}{2}\int_{0}^{T}\Slo(\hat \rho_t^\tau, \hat \mu_t^\tau)  \d t \leq 	F (\rho_{0}^{\tau}, \mu_{0}^{\tau}).
		\end{align} 
		
		Lemma \ref{lem:limiti curve} gives that the curves obtained by the three interpolations, defined above, for the JKO scheme \eqref{defn:JKO}, converge to the same limit curve $(\rho_t, \mu_t)$. Lemma \ref{lem:sqrt_S_L2H1} uses the convergence of the piecewise constant interpolation to deduce $\sqrt{\rho_t+\mu_t}\in L^2_tH^1_x$ and, together with \eqref{discrete EDI 3} (which provides the uniform bound on $\Slo$) and Proposition \ref{prop:F_l.s.c._t} we obtain $(\rho,\mu)\in L^2\mc H$ and $\int_0^T \Slo(\rho_t, \mu_t) \leq \liminf_\tau \int_0^T \Slo(\hat \rho_t^\tau, \hat \mu_t^\tau) $.

		Let us now look at the momentum variables $\tilde E_t$. We consider the Benamou-Brenier functional 
		$$\mc B(\rho,E):=\begin{cases}
			\int_\O \rho|v|^2 \d x &\mbox{ if }E=\rho v,v\in L^2(\rho),\\
			+\infty & \mbox{ otherwise}.
		\end{cases}$$
		This functional, see Chapter 5 in \cite{S15}, is lower semi-continuous for the weak convergence of both variables. Hence, we can write 
		$$\frac{1}{2 } \int_{0}^{T}\int_\O\tilde \rho_t^\tau \vert   \tilde{v_t}^\tau\vert^2 \d x \d t + \frac{1}{2 }  \int_{0}^{T}\int_\O  \tilde \mu_t^\tau \vert \tilde{w_t}^\tau\vert^2 \d x \d t = \mc B(\tilde\rho^\tau,\tilde E_\rho^\tau)+\mc B(\tilde\mu^\tau,\tilde E_\mu^\tau),$$
		and pass to the limit as $\tau\to 0$, thus deducing $E_\rho=\rho v$ with $v\in L^2(\rho)$ and $E_\mu=\mu w$ with $w\in L^2(\mu)$ (all these integrabilities being meant in space-time) and 
		$$\frac{1}{2 } \int_{0}^{T}\int_\O \rho_t \vert  {v_t}\vert^2 \d x \d t + \frac{1}{2 }  \int_{0}^{T}\int_\O   \mu_t \vert {w_t}\vert^2 \d x \d t \leq \liminf_{\tau \to 0} \frac{1}{2 } \int_{0}^{T}\int_\O\tilde \rho_t^\tau \vert   \tilde{v_t}^\tau\vert^2 \d x \d t + \frac{1}{2 }  \int_{0}^{T}\int_\O  \tilde \mu_t^\tau \vert \tilde{w_t}^\tau\vert^2 \d x \d t. $$
		We then combine this with the lower semi-continuity of $F$ which is implied by the convexity of $f$, and we see that the limit curves $(\rho,\mu)$ together with the velocity fields $(v,w)$ is indeed an EDI solution of our PDE. 
	\end{proof}
	
	\paragraph{An additional estimate.}
	
	We conclude this section by showing an additional estimate on $S$, which is actually not needed for our analysis.
	\begin{lemma}  \label{coroll_extra}
		Assume that the problem we consider is in dimension 
		$1$ or $2$, that is, $\Omega \subset \R^d$ with $d=1$ or $d=2$. Then, for $T>0$, we have $ \int_{0}^{T} \|S_t\|_{L^2(\Omega)}^2 \d t < + \infty$.
	\end{lemma}
	\begin{proof} We treat the cases $d=1$ and $d=2$ separately. 
		
		By Lemma \ref{lem:sqrt_S_L2H1}, we have that $\sqrt{S} \in L^2_t H_x^1$ and this implies $\sqrt{S} \in L^2_t L_x^{\infty}$, thus  $S \in L_t^1 L_x^{\infty}$. Moreover, since $S=\rho+\mu$ and $\rho$ and $\mu$ have unit mass, we also have $S \in L_t^\infty L^1_x$.
		Then, we have 
		\begin{align*}
			\int_0^T  \|S_t\|^2_{L^2(\O)} \d t = \int_0^T \int_{\Omega} |S_t(x)|^2 \d x \d t  \leq  	 \int_0^T \|S_t\|_{L^\infty (\Omega)} \|S_t\|_{L^1(\Omega)} \d t < + \infty.
		\end{align*} This proves the result for $d=1$. 
		
		\medskip 
		Now, we consider a function $\phi(b) = b \left(\log b -c \right)$, and notice that the Legendre transform $\phi^*$ of $\phi$ is given by $\phi^*(a) = e^{a+c-1}$. Then we have the following inequality for every $a,b, c$ (with $b>0$):
		\begin{align}  \label{extra_Legendre}
			b  (\log b - c) + e^{a+c-1} \geq ab.
		\end{align}
		Let us define $h(t) := \|\sqrt{S}_t\|^2_{H^1(\Omega)}  < +\infty$ and notice that $h(t)\in L^1([0,T])$. Taking $b(t,x) = S_t(x) h(t)$ and $a(t,x) = \frac{S_t(x)}{h(t)}$, together with a function $c=c(t)$ to be chosen in \eqref{extra_Legendre}, we obtain
		\begin{multline} \label{extra_est1}
			\int_0^T \| S_t\|_{L^2(\Omega)}^2 \d t = \int_0^T \int_{\Omega} \vert S_t(x)\vert^2 \d x \d t = 	\int_0^T \int_{\Omega} S_t(x) h(t) \frac{S_t(x)}{h(t)}  \d x \d t 
			\\ \leq \int_0^T \int_{\Omega}  S_t(x) h(t) \left ( \log (S_t(x) h (t))  - c(t) \right  ) + e^{\frac{S_t(x)}{h(t)}+c(t)-1}  \d x \d t .
		\end{multline} 
		Now let us recall the Moser-Trudinger inequality in dimension $d=2$. There exist positive constants $\alpha  $ and $C$ such that for every $u \in H^1 (\Omega)$ with $\|u\|_{H^1(\Omega)} \leq 1$ we have
		\begin{align} \label{extra_MT}
			\int_{\Omega} e^{\alpha|u|^2 } \d x \leq C.
		\end{align}
		The sharp value of the constant $ \alpha$ is  $ 4 \pi$, but we will just use $\alpha=1$. Taking $c (t)= \log h(t)$  and denoting $u (t,x)= \sqrt{\frac{S_t(x)}{h(t)}}$ in \eqref{extra_est1} we obtain that
		\begin{align*} 
			\int_0^T \| S_t \|^2_{L^2(\O)} \d t \leq \int_0^T h(t)\int_{\Omega} S_t(x) \log (S_t(x)) \d x \d t  + e^{-1}\int_0^T h(t) \d t \int_{\Omega}  e^{u^2} \d x  \lesssim \| h \|_{L^1([0,T])} < +\infty.
		\end{align*}
		The very last inequality also relies on the fact that $S$ has uniformly bounded entropy, i.e. $\int_\O S_t\log S_t \d x\leq C$. This is a consequence of the bound on $F(\rho_t,\mu_t)\leq F(\rho_0,\mu_0)$. Using $F\geq G$ we have a uniform bound on $\int \rho\log\rho+\mu\log\mu$ and, by convexity, on $\int \frac S2\log(\frac S2)$, which in turn gives a bound on the entropy of $S$. This gives the result for $d=2$ and finishes the proof.
	\end{proof}

	\section{Differentiation properties}
	\label{sec:differentiation_properties}
	The goal of this section is to prove a statement similar to the following one:
	
	\medskip 
	
	{\it Let $(\rho,\mu)$ be a curve in $L^2\mathcal H$ (see Def. \ref{defn:L^2H}) and let $v$ and $w$ be two velocity fields for $\rho$ and $\mu$, respectively, i.e. we have $\partial_t\rho+\nabla\cdot(\rho v)=0$ and $\partial_t\mu+\nabla\cdot(\mu w)=0$. 
		
		If $g:\R^2_+\to\R$ is a ``nice enough'' function, we have
		\begin{equation}\label{goaleq}
			\int_\Omega g(\rho_T,\mu_T)=\int_\Omega g(\rho_0,\mu_0)+\int_0^T\int_\O \rho_t\nabla g_a(\rho_t,\mu_t)\cdot v_t +\mu_t \nabla g_b(\rho_t,\mu_t)\cdot w_t.
		\end{equation}
		
		In particular, we would like this to be true for $g=f$ and for $(\rho,\mu)$ the solution that we found in Section \ref{sec:EDI}.
	}

	\medskip
	
	The main idea behind the above computation is that \eqref{goaleq} holds if the densities of $\rho$ and $\mu$ are smooth, by just using a differentiation under the integral sign and an integration by parts. Hence, we will prove the result by regularization, relying on a suitable convolution kernel (in space only). We will suppose that $\O$ is either the torus or a regular cube. In the second case, after symmetrizing, the functions $\rho$ and $\mu$ can be extended by periodicity and it is exactly as if $\O$ were the torus.
	
	We first choose a convolution kernel $\eta>0$ with $\int \eta=1$ (some assumptions on it will be made precise later) and define $\eta_\varepsilon$ to be its rescaled version $\eta_\varepsilon(z) :=\varepsilon^{-d}\eta(z/\varepsilon)$. We define $\rho_\ve$ and $v_\ve$ via
	\begin{align} \label{convolved}
		\rho_\ve(t,\cdot)=\eta_\ve*\rho_t, \quad \mbox{and } \quad (\rho_\ve v_\ve)(t,\cdot)=\eta_\ve*(\rho_t v_t).
	\end{align}
	Analogously, we define $\mu_\ve$ and $w_\ve$ with the same convolution kernel.
	
	\medskip
	
	We first observe the following property.
	
	\begin{lemma}
		If $\rho\in L^1([0,T]\times\O)$ and we have $\int_0^T\int_\O \rho|v|^2<+\infty$, and $\rho_\ve$ and $v_\ve$ are defined as \eqref{convolved}, then we have $\sqrt{\rho_\ve}v_\ve\to \sqrt{\rho}v$ in $L^2([0,T]\times\O)$.
	\end{lemma}
	
	\begin{proof}
		It is well known (see Chapter 5 in \cite{S15}) that we have 
		$$\int_0^T\int_\O \rho_\ve|v_\ve|^2\leq \int_0^T\int_\O \rho|v|^2,$$
		which proves that $\sqrt{\rho_\ve}v_\ve$ is bounded in $L^2$. Moreover, it is clear that its pointwise limit is $\sqrt{\rho}v$ on the set $\{\rho>0\}$ as a consequence of the standard properties of the convolution and of $\rho_\ve\to \rho$ and $\rho_\ve v_\ve\to \rho v$. We then use Lemma \ref{pointwiseweak} below to deduce the strong $L^2$ convergence.
	\end{proof}
	
	\begin{lemma}\label{pointwiseweak}
		Suppose that a sequence $u_n\in L^2(X;\R^d)$ weakly converges in $L^2$ to a function $v$,  and that we have $u_n(x)\to u(x)$ for a.e. $x\in A\subset X$. Then we have $u=v$ on $A$.
		
		Suppose that a sequence $u_n$ satisfies $\limsup_n \int |u_n|^2 \leq \int |u|^2$ and $u_n(x)\to u(x)$ for a.e. $x\in A$ with $A=\{u\neq 0\}$; then we have $u_n\to u$ in $L^2(X)$.
	\end{lemma}

	\begin{proof}
		Let $\phi\in L^\infty(X)$ be a test function vanishing on $A^c$. We have $\int u_n \cdot \phi\to \int v\cdot \phi$ because of weak convergence. Yet, for an arbitrary $R>0$, if we denote by $\pi_R$ the projection onto the closed ball of radius $R$, we also have $\int \pi_R(u_n)\cdot\phi\to\int\pi_R(u)\cdot\phi$ because of dominated convergence. Moreover 
		\begin{align*}
			\int_X |u_n-\pi_R(u_n)| \vert \phi \vert\leq ||\phi||_{L^\infty}\int_{\{|u_n|>R\}} |u_n|\leq \frac{||\phi||_{L^\infty}||u_n||_{L^2}^2}{R}\leq \frac{C}{R},
		\end{align*}
		
		\begin{align*}
			\int_X |u-\pi_R(u)| \vert \phi\vert\leq \frac{C}{R}.
		\end{align*} Using
		$$\int_X (u_n-u)\cdot\phi\leq \int_X |u_n-\pi_R(u_n)||\phi|+\int_X |u-\pi_R(u)||\phi|+\int_X (\pi_R(u)-\pi_R(u_n))\cdot\phi,$$
		we obtain
		$$\left|\limsup_{n} \int_X (u_n-u)\cdot\phi\right|\leq \frac{2C}{R},
		$$
		i.e. $\lim_n \int_X (u_n-u)\cdot\phi=0$ since $R$ is arbitrary. We then obtain $\int (u-v)\cdot\phi=0$ for any $L^\infty$ function $\phi$ vanishing outside $A$, i.e. $u=v$ a.e. on $A$.
		
		\medskip 
		
		For the second part of the statement, since $u_n$ is bounded in $L^2$ we first extract a subsequence weakly converging to some $v\in L^2$. From the previous part of the claim and our assumptions, we know that $v=u$ on $A=\{u\neq 0\}$. We then use
		$$\int_X |v|^2\leq \liminf_n \int_X |u_n|^2\leq \limsup_n \int_X |u_n|^2\leq\int_X |u|^2=\int_A |u|^2=\int_A |v|^2\leq \int_X |v|^2.$$
		Since all the inequalities must be equalities, we deduce $v=0=u$ on $A^c$ and hence $u=v$ a.e. on $X$, as well as $||u_n||_{L^2}\to ||u||_{L^2}$. Thus, $u_n$ weakly converges to $u$ with convergence of the norm, which implies strong convergence. The limit does not depend on the subsequence we extracted, so it holds on the full sequence.
	\end{proof}
	
	With this convergence result in mind we can first prove the following:
	\begin{prop}
		Equality \eqref{goaleq} holds when $g(a,b)=\tilde f(a+b)$, provided that $\tilde f$ (which is a priori only defined on $[2,+\infty)$) is extended to a function bounded from below on $\R_+$ satisfying $0\leq\tilde f''(s)\leq C/s$.
	\end{prop}
	\begin{proof}
		We first regularize by convolutions the densities $\rho$ and $\mu$ into $\rho_\ve$ and $\mu_\ve$. We define
		\begin{align*}
			S_\ve:=\rho_\ve+\mu_\ve=S*\eta_\ve, 
		\end{align*} where $S=\rho+\mu$. Note that $\sqrt{S}\in L^2_tH^1_x$ implies  $\nabla(\sqrt{S_\ve})\to\nabla(\sqrt{S})$ in $L^2$ in space-time. Indeed, a simple convexity argument proves  that $||\nabla(\sqrt{S_\ve})||_{L^2}\leq ||\nabla(\sqrt{S})||_{L^2}$. This proves that $\nabla(\sqrt{S_\ve})$ is bounded in $L^2$ and has, hence, a weak limit. This limit can only be $\nabla(\sqrt{S})$ since $\sqrt{S_\ve}\to\sqrt{S}$ pointwise. Yet, this also means that the limit of the $L^2$ norm will be smaller than the norm of the limit, which implies strong convergence (even if, actually, in this context weak convergence would have been enough). 
		
		Since, for the regularized densities, formula \eqref{goaleq} is true, we just have to pass to the limit each term as $\ve\to 0$. The convergence 
		$$\int_\O g(\rho_\ve,\mu_\ve)\to \int_\O g(\rho,\mu)$$
		is true for any convex function $g$ which is bounded from below by combining a Fatou's lemma giving the lower bound on the liminf and a Jensen inequality giving an upper bound (indeed, we use the fact that every convex functional invariant by translations decreases by convolution). 
		
		We then just need to pass to the limit the integral in space-time. The part concerning $\rho$ can be written as
		$$\int_0^T\int_\Omega \rho_\ve \tilde f''(S_\ve)\nabla S_\ve\cdot v_\ve=\int_0^T\int_\Omega \sqrt{\rho_\ve S_\ve} \tilde f''(S_\ve)\frac{\nabla S_\ve}{\sqrt{S_\ve}}\cdot (\sqrt{\rho_\ve}v_\ve).$$
		In the last expression, we notice that $\sqrt{\rho_\ve S_\ve} \tilde f''(S_\ve)$ is bounded (since $\sqrt{\rho_\ve S_\ve}\leq S_\ve$ and $\tilde f''(s)\leq C/s$) and pointwisely converges to $\sqrt{\rho S} \tilde f''(S)$. This implies that, if we multiply by $\frac{\nabla S_\ve}{\sqrt{S_\ve}}=2\nabla(\sqrt{S_\ve})$, we still have a strong $L^2$ convergence to $\sqrt{\rho S} \tilde f''(S)\frac{\nabla S}{\sqrt{S}}$. Together with the strong $L^2$ convergence $\sqrt{\rho_\ve}v_\ve\to \sqrt{\rho}v$, the result is proven.
	\end{proof}
	
	The next step requires to discuss the following property.
	
	\begin{definition}
		We say that a convolution kernel $\eta$ satisfies Property {\bf H1 conv} if the following holds: There exists a constant $C$ such that, for every function $u\in L^1$ with $u_+\in H_1\cap L^\infty$ and every positive constant $c>0$ we have
		$$||\nabla((u*\eta_\ve-c)_+)||_{L^2}\leq C\frac{||u_+||_{L^\infty}}{c}||\nabla(u_+)||_{L^2}.$$
	\end{definition} 
	\begin{prop} \label{prop goaleq_in_B}
		Suppose that there exists a convolution kernel $\eta$ satisfying Property {\bf H1 conv}. Then, Equality \eqref{goaleq} holds when $g\in C^2$ is compactly supported in $B\subset\R_+^2$.
	\end{prop}
	
	\begin{proof}
		We will proceed as above by convolution, but we choose a convolution kernel $\eta$ satisfying Property {\bf H1 conv}.
		
		The proof of this proposition will recall a lot the different steps in that of Proposition \ref{prop cv B}. Let us take a ``triangle'' function $T_{\alpha,\beta,c}$, as defined in \eqref{T}, whose support $\{(a,b)\in\R_+^2\,:\, \alpha a + \beta b \leq c\}$ is contained in the open set $B$. We can assume that there exists another such triangle contained in $B$ with coefficients $\alpha,\beta,c'$ with $c'>c$.
		
		We set $u=c'-(\alpha\rho+\beta\mu)$ and we observe that $u_+\in L^2_t H^1_x$ because $(\rho,\mu)\in L^2\mathcal H$. We also have $u\leq c'$ which shows that $u_+$ also belongs to $L^\infty$. We then deduce that the functions $(c-\alpha \rho_\ve-\beta\mu_\ve)_+=(u*\eta_\ve-(c'-c))_+$ are bounded in $L^2_t H^1_x$ because of Property {\bf H1 conv}. 
		
		We deduce that 
		$$\nabla((c-\alpha \rho_\ve-\beta\mu_\ve)_+)=-\ind_{\{\alpha\rho_\ve+\beta\mu_\ve<c\}}(\alpha\nabla\rho_\ve+\beta\nabla\mu_\ve)
		$$ weakly converges to 
		$$\nabla((c-\alpha \rho-\beta\mu)_+)=-\ind_{\{\alpha\rho+\beta\mu<c\}}(\alpha\nabla\rho+\beta\nabla\mu).
		$$ If we fix a continuous function $\chi$, supported in the support of $T_{\alpha,\beta,c}$, we have the  pointwise convergence $$
		\chi(\rho_\ve,\mu_\ve)\to\chi(\rho,\mu),
		$$ and this implies the weak $L^2$ convergence
		$$
		\chi(\rho_\ve,\mu_\ve)(\alpha\nabla\rho_\ve+\beta\nabla\mu_\ve)\deb\chi(\rho,\mu)(\alpha\nabla\rho+\beta\nabla\mu).
		$$
		This is true for any continuous function $\chi$ supported in the triangle which is the support of $T_{\alpha,\beta,c}$. Yet, $\chi$ is also supported in different triangles, that we can obtain as supports of other functions of the form $T_{\alpha,\beta,c}$, by slightly changing the values of $\alpha,\beta,c$. Therefore, we deduce that, for each such continuous function $\chi$, we have
		$$\chi(\rho_\ve,\mu_\ve)\nabla\rho_\ve\deb\chi(\rho,\mu)\nabla\rho\quad\mbox{and}\quad \chi(\rho_\ve,\mu_\ve)\nabla\mu_\ve\deb\chi(\rho,\mu)\nabla\mu.$$
		Following the argument of the proof of Lemma \ref{cv grad weak}, i.e., summing up many such functions $\chi$ and using partition of the unity, we can deduce that the same result finally holds for any continuous function $\chi$ supported in $B$ (again, since the domain $B$ is a union of such triangles).
		
		We now proceed in the usual way with the approximation by convolution, since we just need to observe that we have
		$$\int_0^T\int_\O \rho_\ve\nabla g_a(\rho_\ve,\mu_\ve)\cdot v_\ve =\int_0^T\int_\O \sqrt{\rho_\ve}(g_{aa}(\rho_\ve,\mu_\ve)\nabla\rho_\ve+g_{ab}(\rho_\ve,\mu_\ve)\nabla\mu_\ve)\cdot \sqrt{\rho_\ve}v_\ve,$$
		and then use the strong convergence of $\sqrt{\rho_\ve}v_\ve$ and the weak convergence of the other term, since $g_{aa}$ and $g_{ab}$ are continuous and compactly supported in $B$.
		
		We also need to pass the terms $\int g(\rho_\ve(0),\mu_\e(0))$ and $\int g(\rho_\ve(T),\mu_\e(T))$ to the limit, but this can be easily done by dominated convergence since $g$ is bounded.
	\end{proof}
	
	The next results extend the above proposition to the case which is of interest for us. For $\delta>0$, we set 
	$$
	B_\delta =\{(a,b)\in B\,:\, (a+\delta,b+\delta)\in B\}.
	$$
	\begin{prop}\label{gC11}
		Suppose that there exists a convolution kernel $\eta$ satisfying Property {\bf H1 conv}. Then, Equality \eqref{goaleq} holds when $g\in C^{1,1}$ is compactly supported in the open set $B\subset\R_+^2$, with $g\in C^2(B_\delta)$ and $g=0$ on $B\setminus B_\delta$.
	\end{prop}
	\begin{proof}
		The proof is obtained by approximation using a cut-off function $\chi_r$ with the following properties: 
		\begin{itemize}
			\item $0\leq \chi_r\leq 1$, 
			\item $\chi_r(a,b)=1$ if the distance between $(a,b)$ and $B\setminus B_\delta$ is larger than $2r$ and $\chi_r(a,b)=0$ if this same distance is smaller than $r$,
			\item  $\chi_r\in C^2$, with $|\nabla \chi_r|\leq Cr^{-1}$ and $|D^2 \chi_r|\leq Cr^{-2}$.
		\end{itemize}
		We then apply Proposition \ref{prop goaleq_in_B} to $g\chi$, which is $C^2$, and pass to the limit. The convergence of the boundary terms (in time) is obtained by pointwise convergence and dominated convergence. For the convergence of the integral terms in space-time we need to dominate the second derivatives 
		$$
		(g\chi_r)_{ij}=g_{ij}\chi_r+g_i(\chi_r)_j+g_j(\chi_r)_i+g(\chi_r)_{ij}.
		$$ In the above sum, the first term is bounded. The second and the third terms are also bounded since $|\nabla\chi_r | \leq Cr^{-1}$ but $|\nabla g|\leq r$ on $\{\nabla \chi_r\neq 0\}$, which is a consequence of the $C^{1,1}$ behaviour of $g$. Similarly, for the last term we use $|D^2\chi_r | \leq Cr^{-2}$ with $|g|\leq r^2$ on $\{D^2 \chi_r\neq 0\}$.
	\end{proof}
	
	In order to conclude, we now take the function $f$ and write it as $f(a,b)=\bar g(a,b)+\tilde f(a+b)$, with $\bar g\in C^{1,1}$ and $\bar g=0$ on $A$. It is not yet possible to apply Proposition \ref{gC11} to $g=\bar g$ since $\bar g$ is supported on $\overline B=B\cup\partial A$ and not on $B$, but we will obtain this by approximation.
	
	\begin{prop}
		Suppose that there exists a convolution kernel $\eta$ satisfying Property {\bf H1 conv}. Then, Equality \eqref{goaleq} holds when $g=\bar g$, and hence for $g=f$, provided that the curve $(\rho,\mu)$ is such that $\int_0^T \Slo(\rho,\mu) \d t$ is finite.
	\end{prop}
	\begin{proof}
		We first apply  Proposition \ref{gC11} to $g(a,b):=\bar g(a+\delta,b+\delta)$ and obtain
		\begin{multline}\label{bargoaleq}
			\int_\Omega \bar g(\rho_T+\delta,\mu_T+\delta)-\int_\omega \bar g(\rho_0+\delta,\mu_0+\delta)\\
			=\int_0^T\int_\O \rho v\cdot\left( \bar g_{aa}(\rho+\delta,\mu+\delta)\nabla\rho+  \bar g_{ab}(\rho+\delta,\mu+\delta)\nabla\mu\right)\\
			+\int_0^T\int_\O \mu w\cdot \left(\bar g_{ab} (\rho+\delta,\mu+\delta)\nabla\rho
			+  \bar g_{bb}(\rho+\delta,\mu+\delta) \nabla \mu \right).
		\end{multline}
		We have the following formulas for the second derivatives of $\bar g$:
		\begin{eqnarray*}
			\bar g_{aa}(a,b)&=&\left(\frac 1a-\tilde f''(a+b)\right)\ind_{B}(a,b),\\
			\bar g_{ab}(a,b)&=&\left(1-\tilde f''(a+b)\right)\ind_{B}(a,b),\\
			\bar g_{bb}(a,b)&=&\left(\frac 1b-\tilde f''(a+b)\right)\ind_{B}(a,b).
		\end{eqnarray*}
		We need to prove that all the terms computed in $(\rho+\delta,\mu+\delta)$ converge to the corresponding terms in $(\rho,\mu)$. To do so, we consider the difference between the terms with $\delta$ and those without $\delta$. Since $v\in L^2(\rho)$, we just need to prove that the following terms converge to $0$ in $L^2(\rho)$: 
		$$\nabla\rho\left(\frac{1}{\rho+\delta}-\frac1\rho\right)\ind_{B_\delta},\quad \nabla S\left( \tilde f''(S+2\delta)-\tilde f''(S)\right)\ind_{B_\delta},\quad  \nabla f_a(\rho,\mu)\ind_{B\setminus B_\delta},$$
		(we only consider the terms in the first integral, the second integral in \eqref{bargoaleq} is treated similarly, using $w\in L^2(\mu)$).
		The last term is easy to handle, since it converges pointwise to $0$ and it is dominated by $|\nabla f_a(\rho,\mu)|$ which is in $L^2(\rho)$ (owing to the assumption on the slope).
		
		For the term involving $\tilde f''$, we observe that $\tilde f''$ is a Lipschitz function on $[2,+\infty)$ and its extension to $[0,2]$ can be chosen to be Lipschitz as well. In particular, we can impose 
		$$
		|\tilde f '''(S)|\leq C\min\{S^{-2},1\},
		$$ so that we have 
		$$
		\tilde f''(S+2\delta)-\tilde f''(S)\leq C\delta \min\{S^{-2},1\}.
		$$ The $L^2(\rho)$ norm of the desired term is thus smaller than $C\delta ||\min\{S^{-2},1\}\nabla S||_{L^2(S)}$, and the condition $\sqrt{S}\in L^2_tH^1_x$ is more than enough to guarantee the finiteness of the norm which multiplies $\delta$.
		
		We are now left with the term involving $\nabla \rho$. We recall that, setting $X=\frac{\nabla\rho}{\rho}+\nabla\mu$ and 
		$Y=\frac{\nabla \mu}{\mu}+\nabla \rho$, we have $X\in L^2(\rho)$, $Y\in L^2(\mu)$ and $\nabla\rho=\frac{\rho X- PY}{1-P}$, where $P=\rho\mu<1$ on $B$. Hence, the term we are interested in can be re-written as
		$$\frac{\delta\nabla\rho}{\rho(\rho+\delta)}=\frac{\delta(\rho X-P Y)}{\rho(\rho+\delta)(1-P)}.$$
		By distinguishing the case where $(\rho,\mu)$ is close to $(1,1)$ (the only point in $\overline B$ where $P=1$) or far from it we can prove that we have 
		$$\frac{\delta}{(\rho+\delta)(1-P)}\ind_{B_\delta}\leq C.$$
		Indeed, if $(\rho,\mu)$ is far from $(1,1)$ then $1-P$ is bounded from below and we use $\rho+\delta\geq \delta$. If instead $(\rho,\mu)$ is close to $(1,1)$, then $(\rho+\delta)$ is bounded from below, and, moreover, we use $(\rho,\mu)\in B_\delta$ to deduce $(\rho+\delta)(\mu+\delta)<1$, which implies $P+\delta S<1$, hence $1-P\geq\delta S\geq \delta $ (in the last inequality we use again the fact that $(\rho,\mu)$ is close to $(1,1)$ to bound $S$ from below by $1<2$).
		So, we can bound the term we are interested in by $|X-\mu Y|\ind_{B}$. The vector field $X$ is supposed to belong to $L^2(\rho)$, while $Y\in L^2(\mu)$ implies $\mu Y\in L^2(\rho)$ since its squared norm is given by $\int \rho\mu^2|Y|^2$ and $\rho\mu\leq 1$ on $B$.
		This proves that the term involving $\nabla\rho$ is dominated by a term in $L^2(\rho)$ and its pointwise convergence to $0$ shows that it converges to $0$ in $L^2(\rho)$. 
		
		This proves that we can take the limit $\delta\to0$ in the first part of the integral. For the second part, one has to do the same estimates in $L^2(\mu)$, and the computations are the same.
	\end{proof}
	
	\section{An \texorpdfstring{$H^1$}{H^1} estimate on the positive part}
	\label{sec:H1_estimate}
	
	The goal of this section is to prove that, in dimension $d=1$, there indeed exists a convolution kernel satisfying the {\bf H1 conv} property. In order to avoid boundary issues, the result will be proven on the $1-$dimensional torus $\mathbb S^1$. As we already explained in the previous section, this implies a similar result on a segment, after one reflection.
	
	Before going to dimension one, we would like to explain why this very part of the paper unfortunately requires us to restrict ourselves to dimension one: the reason lies in the very different behavior of $H^1$ functions in terms of pointwise bounds. Indeed, in the one-dimensional case $H^1$ functions are continuous, so that the ``bad'' region where $u<0$ (i.e. the region where we only have an $L^1$ control on the function) has to be far away from regions where $u>c$; this is not the case in higher dimension. We note that it would be possible to obtain some estimates similar to the ones defining {\bf H1 conv} in any dimension if we added the assumption that $u \in L^\infty$ but it is not possible to obtain such estimates if we stick to the the $L^1$ assumption that we required which, by density, is actually equivalent to working with measures. 
	
	Unfortunately, for the sake of the applications of {\bf H1 conv} to Section 5, we cannot assume $u\in L^\infty$ as this would be equivalent to $\rho,\mu\in L^\infty$, and $L^\infty$ bounds are unknown in our cross-diffusion system.
	
	Our $1-$dimensional result will be proven by considering the kernel
	$$
	\tilde\eta(x) = \frac{2}{(m-1)(1+\vert x \vert)^m},
	$$
	for $m>2$.  We have $\tilde\eta \in L^1(\R)$, $\int \tilde\eta(x)\d x=1$, and $\tilde\eta$ has a finite first moment. We then define a kernel on the torus using
	$$\eta(x)=\sum_{k\in\mathbb Z} \tilde\eta(x-k).$$
	An easy computation shows that, for $x\neq 0$, we have
	$$
	\tilde\eta^\prime(x) = -\frac{m}{1+\vert x \vert } \tilde \eta(x) \sign(x),
	$$
	which implies in particular 
	\begin{equation}\label{eta'eta}
		\vert \tilde\eta^\prime \vert \leq m \tilde\eta.
	\end{equation}
	Moreover, we have in the distributional sense,
	$$
	\tilde\eta^{\prime\prime}= \frac{m(m+1)}{(1+\vert x \vert)^2}\tilde\eta- \frac{2m}{m-1} \delta_0,
	$$
	hence the positive part of the second derivative satisfies
	$$
	(\tilde\eta^{\prime\prime})_+= \frac{m(m+1)}{(1+\vert x \vert)^2}\tilde\eta.
	$$
	In particular, we also obtain
	$$	(\tilde\eta^{\prime\prime})_+\tilde\eta\geq (1+\frac 1m)	|\tilde\eta^{\prime}|^2.$$

	All these properties can be translated in terms of the kernel $\eta$ on $\mathbb S^1$, since we have uniform convergence of the series defining $\eta$ together with its first and second derivatives. Moreover
	\begin{eqnarray*}
		|\eta^\prime(x)|\leq\sum_k |\tilde\eta^\prime(x-k)|&\leq& \left (1+\frac 1m \right )^{-1/2} \sum_k \vert  (\tilde\eta^{\prime\prime})_+(x-k)^{1/2}\tilde\eta(x-k)^{1/2} \vert \\&\leq&  \left (1+\frac 1m \right )^{-1/2} \left(\sum_k (\tilde\eta^{\prime\prime})_+(x-k)\right)^{1/2}\left(\sum_k \tilde\eta(x-k)\right)^{1/2},   \\
		&=& \left (1+\frac 1m \right )^{-1/2}(\eta^{\prime\prime})_+(x)^{1/2}\eta(x)^{1/2}.
	\end{eqnarray*}
	which proves that also $\eta$, which is a $C^2$ function except at $0$ where $\eta''$ is a negative Dirac mass, satisfies
	$$\left (\eta^{\prime\prime} \right )_+\eta\geq \left (1+\frac 1m \right )	|\eta^{\prime}|^2.$$
	In order to approximate the identity on the torus, we define
	$$
	\tilde\eta_\e(x) = \frac{1}{\e }\tilde\eta \left (\frac{x}{\e}\right ),
	$$
	and $\eta_\ve(x):=\sum_{k\in\mathbb Z}\tilde\eta_\ve(x-k)$. Observe that $\eta_\e$ is of mass $1$ on the torus and that it satisfies 
	$$|\eta_\ve^\prime|\leq \frac \ve \eta_\ve,\quad \left (1+\frac 1m \right ) |\eta_\ve'|^2\leq (\eta_\e^{\prime\prime})_+\eta_\ve.$$

	As a consequence of the properties of $\eta_\ve$ we do have, for any $L^1$ function $w\geq 0$, denoting $w_\ve:=\eta_\ve*w$,
	\begin{equation}\label{wepweppweppp}
		\left (1+\frac1m \right )|w_\ve^\prime|^2\leq w_\ve \,(\eta_\ve^{\prime\prime})_+*w.
	\end{equation}
	This can be, indeed, obtained from the Cauchy-Schwartz inequality:
	\begin{eqnarray*}
		\vert \eta^\prime_\e* w\vert^2
		&\leq&
		\left(\int w(x-y)\vert \eta^\prime_\e(y)\vert \d y\right)^2
		\leq
		\left(\int w(x-y)\left (1+\frac 1m \right )^{-1/2}\sqrt{\eta_\ve(y)(\eta^{\prime\prime}_\e)_+(y)} \ \d y\right)^2\\
		&	\leq& \left (1+\frac 1m \right )^{-1}\int w(x-y)\eta_\ve(y)\d y\int w(x-y)(\eta^{\prime\prime}_\e)_+(y) \d y.
	\end{eqnarray*}
	
	\begin{prop} \label{prop6.1}
		There exists $K>0$ (depending only on $m$) such that, for any $u\in L^1$ whose positive part $u_+$ is in $H^1\cap L^\infty$, and for every $c>0$, we have
		$$\|((u_\e-c)_+)^\prime\|_{L^2}\leq K  \frac{\|u_+\|_{L^\infty}}{c}\|(u_+)^\prime\|_{L^2}.$$
	\end{prop}
	
	\begin{proof}
		Let $u$ be as in the statement of the proposition. For notational simplicity, let us denote $v= u_+$ and $w= u_-$ the positive and negative parts of $u$ respectively. Hence
		$
		u = v - w.
		$
		We have
		\begin{align}\label{dec1}
			\begin{split}
				\int\vert( u_\e -c )_+^\prime \vert^2 &=\int  \mathbbm{1}_{v_\e - w_\e > c} \vert  (v_\e - w_\e)^\prime\vert^2  \\
				&\leq
				2\int \mathbbm{1}_{v_\e - w_\e > c}(\vert  v_\e^\prime \vert^2 + \vert w_\e^\prime\vert^2)
				\leq 2 \int \vert v^\prime  \vert^2 + 2 \int \mathbbm{1}_{v_\e> w_\e +c}\vert  w_\e^\prime\vert^2.
			\end{split}
		\end{align}
		Note that the desired inequality is straightforward if $c>\|v\|_{L^\infty}$ (since in this case $(u_\ve-c)_+=0$), so we can assume $\|v\|_{L^\infty}\geq c$ and estimate the term $\int \vert v^\prime  \vert^2$ with  $\frac{\|v\|_{L^\infty}^2}{c^2} \|v'\|_{L^2}^2$.
		
		We now  handle the last term
		\begin{equation}\label{eq last term}
			\int \mathbbm{1}_{v_\e> w_\e +c}\vert w_\e^\prime\vert^2 =  \int \mathbbm{1}_{\substack{v_\e> w_\e +c\\ \vert v_\e - v\vert > \frac{c}{2}}}\vert w_\e^\prime\vert^2 +  \int \mathbbm{1}_{\substack{v_\e> w_\e +c\\ \vert v_\e - v\vert \leq \frac{c}{2}}}\vert w_\e^\prime\vert^2.
		\end{equation}
		First, observe that, on the set where $w_\e+c <v_\e$, we have $w_\e \leq \|v\|_{L^\infty}$. Moreover, by the definition of $\eta_\e$, we have, on the same set,
		$$
		\vert w_\e^\prime \vert^2 \leq \frac{m^2}{\e^2}w_\e^2 \leq  \frac{m^2}{\e^2}\|v\|_{L^\infty}^2,
		$$
		where the first inequality is a consequence of \eqref{eta'eta}, after scaling by $\varepsilon$.
		Observe also that we have
		$$
		\left\vert \left \{ \vert v_\e -v \vert > \frac{c}{2} \right \} \right\vert \leq K(m)\frac{4}{c^2}\e^2 \| v^\prime\|_{L^2}^2.
		$$
		The last inequality is readily obtained by using $	\left\vert \left \{ \vert v_\e -v \vert > \frac{c}{2} \right \} \right\vert \leq \frac{4}{c^2}\int \vert v_\e -v\vert^2$, together with $||v_\ve-v||_{L^2}\leq C||v^\prime||_L^2$, an inequality which is a consequence of the characterization of $H^1$ functions in terms of the norms of the difference between the function and its translations, as soon as the convolution kernel has finite first moment (which leads, hence, to a constant depending on $m$ appear).
		
		Combining the last two estimates we obtain
		$$
		\int \mathbbm{1}_{\substack{v_\e> w_\e +c\\ \vert v_\e - v\vert > \frac{c}{2}}}\vert w_\e^\prime\vert^2 \leq \frac{K(m)}{c^2} \|v^\prime\|_{L^2}^2 \| v\|_{L^\infty}^2.
		$$
		Now, observe that we have
		\begin{align*}
			\int \mathbbm{1}_{\substack{v_\e> w_\e +c\\ \vert v_\e - v\vert \leq \frac{c}{2}}}\vert w_\e^\prime\vert^2 \leq  \int \mathbbm{1}_{\substack{ w_\e \leq \| v \|_{L^\infty} \\  v \geq \frac{c}{2}}}\vert w_\e^\prime\vert^2 \leq  \int \left(\frac{2 v}{c}\right)^2     \left( \frac{\|v\|_{L^\infty}}{w_\e}\right)^2  \vert w_\e^\prime\vert^2
			=   \frac{4 \|v\|_{L^\infty}^2}{c^2} \int v^2 \frac{1}{w_\e^2}  \vert w_\e^\prime\vert^2.
		\end{align*}
		Thus, \eqref{eq last term} finally re-writes as
		\begin{equation}\label{est 1}
			\int \mathbbm{1}_{v_\e> w_\e +c}\vert w_\e^\prime\vert^2 \leq  K(m) \frac{\| v^\prime\|_{L^2}^2 \| v\|_{L^\infty}^2 }{c^2}+ \frac{4 \|v\|_{L^\infty}^2}{c^2} \int v^2      \frac{1}{w_\e^2}  \vert  w_\e^\prime\vert^2.
		\end{equation}
		Let us now treat the following term of \eqref{est 1}:
		$$
		\int v^2  \frac{1}{w_\e^2}  \vert w_\e^\prime\vert^2.
		$$
		We have
		\begin{align}\label{H1conv1}
			\begin{split}
				-\int v^2\frac{1}{w_\e^2} \vert w_\e^\prime\vert^2 = \int v^2 \left (\frac{1}{w_\e} \right )^\prime w_\e^\prime &= -2 \int v v^\prime \frac{ w_\e^\prime}{w_\e}- \int v^2\frac{1}{w_\e}w_\e^{\prime\prime} \\
				&\leq \frac{1}{2m}\int v^2\frac{1}{w_\e^2}\vert w_\e^\prime\vert^2+ 2m\int \vert  v^\prime\vert^2- \int v^2\frac{1}{w_\e}(\eta_\e^{\prime\prime})_+* w.
			\end{split}
		\end{align}
		In the last line we omitted the term with $(\eta_\ve^{\prime\prime})_-*w$ since it is equal to a multiple of $w$, and $vw=0$ (since $v$ and $w$ are the positive and the negative parts of the same function). 
		
		Let us handle the last term on the right-hand side of the above inequality. We have
		$$
		\int v^2\frac{1}{w_\e}(\eta_\e^{\prime\prime})_+* w = \int v^2\frac{1}{w_\e^2}(w* \eta_\e) \left((\eta_\e^{\prime\prime})_+* w\right)\geq \frac{1+m}{m} \int v^2\frac{1}{w_\e^2} \vert   w^\prime_\e\vert^2,
		$$
		where we used \eqref{wepweppweppp} to establish the last inequality.
		Putting this in \eqref{H1conv1}, we get
		$$
		\int v^2\frac{1}{w_\e^2} \vert  w^\prime_\e\vert^2
		\leq 4m^2\int \vert  v^\prime\vert^2.
		$$
		Now, putting this into \eqref{est 1} and into \eqref{dec1} yields the final relation.
	\end{proof}

	\section*{Acknowledgements}
	RD was supported by the IDEXLyon via the Impulsion project of FS ``Optimal Transport and Congestion Games'' PFI 19IA106UDL and by a French PEPS JCJC 2021 grant.
	HY was supported by the European Research Council (ERC) under the European Union’s Horizon 2020 research and innovation programme (grant agreement No 865711) (until October 2021) and partially by the Vienna Science and Technology Fund (WWTF) with a Vienna Research Groups for Young Investigators project, grant VRG17-014 (from October 2021). FS also acknowledges the support of the Lagrange Mathematics and Computation Research Center. 
	
	Finally, the authors warmly thank Alexey Kroshnin for an interesting counter-example showing that in higher dimension there is no convolution kernel for which the very same result of Proposition \ref{prop6.1} can hold.

	\bibliography{Cross-Diff}

\begin{thebibliography}{10}

\bibitem{AB18}
Abdulaziz Alsenafi and Alethea B.~T. Barbaro.
\newblock A convection-diffusion model for gang territoriality.
\newblock {\em Phys. A}, 510:765--786, 2018.

\bibitem{Amin}
Luigi Ambrosio.
\newblock Minimizing movements.
\newblock {\em Rend. Accad. Naz. Sci. XL Mem. Mat. Appl. (5)}, 19:191--246,
  1995.

\bibitem{AGS08}
Luigi Ambrosio, Nicola Gigli, and Giuseppe Savar\'{e}.
\newblock {\em Gradient flows in metric spaces and in the space of probability
  measures}.
\newblock Lectures in Mathematics ETH Z\"{u}rich. Birkh\"{a}user Verlag, Basel,
  second edition, 2008.

\bibitem{BRYZ21}
Alethea B.~T. Barbaro, Nancy Rodriguez, Havva Yolda\c{s}, and Nicola Zamponi.
\newblock Analysis of a cross-diffusion model for rival gangs interaction in a
  city.
\newblock {\em Commun. Math. Sci.}, 19(8):2139--2175, 2021.

\bibitem{B87}
Yann Brenier.
\newblock Décomposition polaire et réarrangement monotone des champs de
  vecteurs.
\newblock {\em C. R. Acad. Sci. Paris Sér. I Math.}, 305(19):805--808, 1987.

\bibitem{B}
Ha\"{\i}m Brezis.
\newblock {\em Functional analysis, {S}obolev spaces and partial differential
  equations}.
\newblock Universitext. Springer, New York, 2011.

\bibitem{CFSS18}
José~A. Carrillo, Simone Fagioli, Filippo Santambrogio, and Markus
  Schmidtchen.
\newblock Splitting schemes and segregation in reaction cross-diffusion
  systems.
\newblock {\em SIAM J. Math. Anal.}, 50(5):5695--5718, 2018.

\bibitem{DeG93}
Ennio De~Giorgi.
\newblock New problems on minimizing movements.
\newblock In {\em Boundary value problems for partial differential equations
  and applications}, volume~29 of {\em RMA Res. Notes Appl. Math.}, pages
  81--98. Masson, Paris, 1993.

\bibitem{DeGiorgiLSC}
Ennio De~Giorgi, Luigi Ambrosio, and Giuseppe Buttazzo.
\newblock Integral representation and relaxation for functionals defined on
  measures.
\newblock {\em Atti Accad. Naz. Lincei Rend. Cl. Sci. Fis. Mat. Nat. (8)},
  81(1):7--13, 1987.

\bibitem{GS14}
Gonzalo Galiano and Virginia Selgas.
\newblock On a cross-diffusion segregation problem arising from a model of
  interacting particles.
\newblock {\em Nonlinear Anal. Real World Appl.}, 18:34--49, 2014.

\bibitem{Giusti}
Enrico Giusti.
\newblock {\em Direct methods in the calculus of variations}.
\newblock World Scientific Publishing Co., Inc., River Edge, NJ, 2003.

\bibitem{JKO98}
Richard Jordan, David Kinderlehrer, and Felix Otto.
\newblock The variational formulation of the {F}okker-{P}lanck equation.
\newblock {\em SIAM J. Math. Anal.}, 29(1):1--17, 1998.

\bibitem{KM18}
Inwon Kim and Alp\'{a}r~Rich\'{a}rd M\'{e}sz\'{a}ros.
\newblock On nonlinear cross-diffusion systems: an optimal transport approach.
\newblock {\em Calc. Var. Partial Differential Equations}, 57(3):1--40, 2018.

\bibitem{S15}
Filippo Santambrogio.
\newblock {\em Optimal transport for applied mathematicians}, volume~87 of {\em
  Progress in Nonlinear Differential Equations and their Applications}.
\newblock Birkh\"{a}user/Springer, Cham, 2015.

\bibitem{S17}
Filippo Santambrogio.
\newblock \{{E}uclidean, metric, and {W}asserstein\} gradient flows: an
  overview.
\newblock {\em Bull. Math. Sci.}, 7(1):87--154, 2017.

\bibitem{SKT79}
Nanako Shigesada, Kohkichi Kawasaki, and Ei~Teramoto.
\newblock Spatial segregation of interacting species.
\newblock {\em J. Theor. Biol.}, 79(1):83--99, 1979.

\bibitem{Villani}
C\'{e}dric Villani.
\newblock {\em Topics in optimal transportation}, volume~58 of {\em Graduate
  Studies in Mathematics}.
\newblock American Mathematical Society, Providence, RI, 2003.

\end{thebibliography}
\end{document}